\documentclass[10pt,reqno]{amsart}
\usepackage{geometry}
\usepackage{amsmath, amsthm, amssymb, hyperref, color}
\usepackage{graphicx}
\usepackage{caption}
\usepackage[all]{xypic}
\usepackage{verbatim}
\usepackage{chemfig, chemnum}
\usepackage{tikz}
\usepackage[linesnumbered,lined,commentsnumbered]{algorithm2e}
\usepackage{caption}
\usepackage{subcaption}
\usepackage{makecell}
\usepackage{array}
\newcolumntype{?}{!{\vrule width 1pt}}
\usepackage{color}
\newif\ifincludeprevious

\includeprevioustrue

\newtheorem{theorem}{Theorem}
\numberwithin{theorem}{section}
\newtheorem{proposition}[theorem]{Proposition}
\newtheorem{lemma}[theorem]{Lemma}
\newtheorem{corollary}[theorem]{Corollary}

\newtheorem{remark}[theorem]{Remark}
\newtheorem{example}[theorem]{Example}

\theoremstyle{definition}
\newtheorem{algorithm_thm}[theorem]{Algorithm}

\newcommand{\RR}{\mathbb{R}}

\newcommand{\CC}{\mathbb{C}}

\newcommand{\NN}{\mathbb{N}}

\DeclareMathOperator*{\conv}{conv}

\date{}
 
\title{Computing Convex Hulls of Trajectories}
\author{Daniel Ciripoi, Nidhi Kaihnsa, Andreas L\"ohne, and Bernd Sturmfels}

\begin{document}

 \begin{abstract}
 \noindent
 We study the convex hulls of trajectories of polynomial dynamical systems.
 Such trajectories include real algebraic curves. The boundaries of the resulting
 convex bodies are stratified into families of faces.
 We present numerical algorithms for identifying these patches.
 An implementation based on the software Bensolve Tools is given.
This furnishes a key step in computing
 attainable regions of chemical reaction networks.
 \end{abstract}
\maketitle

\section{Introduction}

Dynamics and convexity are ubiquitous in the mathematical sciences,
and they come together in applied questions in numerous ways. We 
explore an interaction of dynamics and convexity that  is motivated by
reaction systems in chemical engineering \cite{FH, MGHGM}. 
Consider an autonomous system of ordinary differential equations 
\begin{equation}
\label{eq:dynamics1}
\frac{d}{dt} x(t) \quad = \quad \phi \bigl(x (t) \bigr) ,
\end{equation}
where $x : \RR \rightarrow \RR^n$ is an unknown function, 
and $\phi : \RR^n \rightarrow \RR^n$ is a given polynomial map.
By the Picard-Lindel\"of Theorem, each 
 initial value problem for (\ref{eq:dynamics1}) has a unique solution on a local interval.
 Although $\phi$ is assumed to be polynomial, for most of the techniques 
 we develop it suffices for $\phi$ to be locally Lipschitz continuous.
  Any starting point $x(0)$ in $\RR^n$ gives rise to a unique trajectory
$\,\mathcal{C}:=\left\{ x(t) \mid t \in [0, a)  \text{ for } a >0  \right\} $. This curve may or may not converge to a stationary point,
and its dynamics can be chaotic. We consider the case where the trajectory $\mathcal{C}$ is bounded. If it is not bounded, we restrict time $t$ to a finite interval. 
 
 Our object of study is  the convex hull in
$\RR^n$ of the trajectory starting from~$t=0$:
\begin{equation} \label{eq:convtraj}
{\rm convtraj}\bigl( x(0) \bigr) \,\,\, := \,\,\,
\conv ( \mathcal{C} ).
\end{equation}
We call this set the {\em convex trajectory} of the point $x(0)$.
By definition, it is the smallest convex set containing the trajectory.
The convex trajectory need not be closed. In that case, we usually
replace ${\rm convtraj}\bigl( x(0) \bigr) $ by its topological closure,
so that the convex trajectory becomes a convex body, that is, a compact convex set with nonempty interior. We ensure full-dimensionality of the convex trajectory by restricting to the space in which this curve~lies. 

This article introduces numerical methods that solve the following problems
for the dynamical system (\ref{eq:dynamics1}) with respect to a given 
starting point $y~=~x(0)$~in~$\RR^n$.
\begin{itemize}
\item[(i)] Compute  $ \,{\rm convtraj}(y)$.  This is a convex body in $\RR^n$.
The output should be in a format that represents the boundary  as accurately as possible.
\item[(ii)] Decide whether  the convex trajectory $\,{\rm convtraj}(y)\,$ is forward closed.
This happens if and only if the vector field $\phi(z)$ at every boundary point $z$
 points inwards, for every 
supporting hyperplane.
\end{itemize}

A subset $S \subset \RR^n$ is {\em forward closed}  if every trajectory 
of the dynamical system (\ref{eq:dynamics1})
that starts in $S$ remains in~$S$. In the literature, forward closed sets are also referred to as forward invariant sets.
If ${\rm convtraj}(y)$ is forward closed then it equals the {\em attainable region}
of the point $y$ for  (\ref{eq:dynamics1}). The attainable region
 is the smallest subset of $\RR^n$ that contains
$y$ and is both convex and forward closed. We refer to
\cite{FH, Kai, MGHGM} for motivation and many details.

This article enlarges the repertoire of {\em convex algebraic geometry},
a field that studies convex semialgebraic sets, and their role in
polynomial optimization \cite{BPT}. Indeed, every algebraic curve is locally the trajectory
of a polynomial dynamical system. Hence, our results
apply to convex hulls of algebraic curves \cite{RS2}.
While Problem (ii) is specific to dynamical systems,
Problem (i) makes sense for any smooth curve that can be approximated sufficiently well. In particular, this is the case for parametric curves $\,t \mapsto x(t)$.
See Figures \ref{fig:trott}, \ref{fig:test2} and~\ref{fig:yellow_green_1}. 

Approximations by polygonal curves
are crucial for our approach. The boundary of the convex hull of a generic curve has a distinct structure that is characterized by families of polyhedral faces spanned by curve points. Our algorithms account for this by constructing polyhedral approximations of convex trajectories, using points that are sampled 
from the curve of interest.
We will discuss conditions on $\mathcal{C}$ and the sampling that assures the significance of our methods. More specifically, our
results establish a connection between the facial structure of the convex bodies we study and the facets of the approximating polytopes.

Our presentation is organized as follows.
Section~\ref{sec2} treats the planar case $(n=2)$.
Example \ref{ex:trott} illustrates our solution for the Hamiltonian system
associated with the Trott curve.
Section~\ref{sec3} develops a general geometric theory for
taking limits of  convex polytopes. This will be applied to
identify the desired convex body  (\ref{eq:convtraj})
from a sequence of inner approximations by polytopes.
We compute these polytopes
using the {\tt Matlab}/{\tt Octave} package {\tt Bensolve Tools}
\cite{CLW}. The optimization methodology that underlies this approach is explained in Section~\ref{sec4}. Section~\ref{sec5} describes a stratification of a piecewise smooth
 convex body of dimension $n$. The strata are $(n-k-1)$-dimensional
 families of $k$-dimensional faces, called {\em patches}, for various $k$.
 Section~\ref{sec6} concerns dynamical systems (\ref{eq:dynamics1}) whose
trajectories are algebraic or trigonometric curves.
This includes linear dynamical systems. Table~\ref{tab:3Ddata}
offers data on their patches for $n=3$.
Section~\ref{sec7} presents our solution to Problem (ii).
Algorithm \ref{alg:redgreen} decomposes each patch of the convex trajectory (\ref{eq:convtraj})  into two subsets,
depending whether the vector field $\phi(z)$ points inward or outward. If
the outward set is always empty then (\ref{eq:convtraj}) is forward closed.
In Section~\ref{sec8} we focus on  chemical reaction networks with
mass action kinetics \cite{HT, JS}.
These motivated our study.
We characterize planar algebraic curves that are trajectories of chemical reaction networks, we study the {\em van de Vusse system} 
\cite{FH, MGHGM}, and we exhibit a toric dynamical system \cite{CDSS} 
with ${\rm convtraj}(y)$ not forward closed. This resolves
  \cite[Conjecture~4.1]{Kai}.

\section{Planar Scenarios}
\label{sec2}

In this section we study the convex trajectories of dynamical systems
in the plane $(n=2)$. Our input is a pair of polynomials
$\phi = (\phi_1,\phi_2)$ along with a starting point $y = (y_1,y_2)$ in $\RR^2$.
The resulting trajectory of (\ref{eq:dynamics1}) is a 
smooth plane curve $\mathcal{C}$ that is parametrized by time~$t$.
We write $C = {\rm conv}(\mathcal{C})$ for the associated convex trajectory.
This is a convex region in $\RR^2$.

The boundary $\partial C$ of the convex trajectory $C$ consists of
arcs on the curve $\mathcal{C}$
and of edges that connect them.
Each edge of $C$ is a line segment between two points on $\mathcal{C}$.
These are either points of tangency or endpoints of the curve.
This partitions $\partial C$ into
{\em patches}, described in general in Section \ref{sec5}. Here,
 $k$-patches are edges (for $k=1$) and arcs (for $k=0$).

If we had an exact algebraic representation of the curve $\mathcal{C}$
then we could use symbolic methods to compute its bitangents
and derive from this a description of $\partial C$. 
For instance, if $\mathcal{C}$ is an algebraic curve  of degree four,
as in Example \ref{ex:trott}, then it has $28$ bitangent lines 
(over $\CC$) which can be computed using Gr\"obner bases. 
In \cite{BK} and \cite{EKKL}, methods for reducing the number of  
bitangent line computations  are introduced for convex hull computations in the planar case.
But, such algebraic representations are not available when we study
dynamical systems. Each 
trajectory is an analytic curve $\,t \mapsto x(t)$.
This parametrization is given indirectly, namely by
the differential equation (\ref{eq:dynamics1}) it satisfies. Furthermore, our aim is not only to obtain a representation for the convex hull of a curve but also to  characterize the structure of its boundary. We are not aware of any reports on such studies.

\begin{algorithm_thm}(Detection of edges and arcs for $n = 2$)

\label{alg:boundary_examining_2}
\begin{algorithm}[H]
\SetKwData{CommonVerts}{CommonVerts}
\SetKwFunction{Conv}{Conv} \SetKwFunction{Graph}{Graph} 
\SetKwFunction{Ben}{Bensolve} \SetKwFunction{AggregateShort}{AggregateShort}
\SetKwFunction{ConnectedComponents}{ConnectedComponents}
\SetKwInOut{Input}{input}\SetKwInOut{Output}{output}
\Input{A list $\mathcal{A}$ of points on a curve $\mathcal{C}$ in $\RR^2$;  a threshold value $\delta > 0$} 
\Output{The numbers $\#_0$ and $\#_1$ of arcs and edges of $C=\conv \left( \mathcal{C} \right)$ \\
For each $i$: \ list of curve points that represent the $i$th arc of $\partial C$. \\
List of line segments that represent the edges of $C$. 
}
Compute the vertices $\mathcal{V}$ and edges $\mathcal{H}$ of $A = \conv \left( \mathcal{A} \right)$. \\
Build a graph $G$ with node set  $\mathcal{H}$ such that two distinct edges $H_1,H_2$ of $A$ form an edge of $G$ if $H_1 \cap H_2 \neq \emptyset$ and both $H_1$ and $H_2$ have length $\leq \delta$.\\ 
Output the number $\#_1$ of isolated nodes of $G$ and the number $\#_0$ of remaining connected components $G_i$.\\ 
\ForEach{nonsingleton connected component $G_i$}{ Output a list of curve points that are endpoints of those edges of $A$, that belong to $G_i$. This represents the $i$th arc of $\partial C$.}
Edges $H_j$ of $A$ that correspond to isolated nodes of $G$ represent edges of $C$.\\
\end{algorithm}
\end{algorithm_thm}

We now assume that a polygonal approximation is given 
for the curve $\mathcal{C}$. Our input is a finite list $\mathcal{A}$ of 
points $x(t_i)$ on $\mathcal{C}$. In our computations we solve 
the differential equation (\ref{eq:dynamics1}) numerically
using the versatile solver {\tt ode45} in {\tt Matlab}.
This generates the set $\mathcal{A}$ of sample points 
which we assume to be reliably accurate. Using {\tt ode45} also allows us to control the quality of the approximation.
To ensure a certain precision of the approximation one could employ better methods for integrating dynamical systems like the one presented in \cite{Hoe}.

As a first step we address Problem (i) in the Introduction. 
Algorithm \ref{alg:boundary_examining_2} computes a representation of the boundary
$\partial C$ of the convex trajectory $C$ from a polygonal approximation $\mathcal{A}$ of the trajectory~$\mathcal{C}$. A key idea is the identification of long edges in $A := \conv\left(\mathcal{A}\right)$.
In Section \ref{sec5} we generalize to curves in $\RR^n$.
Algorithm~\ref{alg:boundary_examining_2} is a special case of Algorithm \ref{alg:boundary_examining_3}.

We next solve Problem (ii) from the Introduction. For each point $z$ on an arc of $C$,
the vector $\phi(z)$ is tangent to the curve, so there is nothing to be checked
at the arcs. To decide whether $C$ is forward closed with respect to the dynamics (\ref{eq:dynamics1}),
we must examine the edges of $C$. Suppose that ${\rm conv} \bigl( x(t_i),x(t_j) \bigr)$ is an edge of $C$,
and consider its relative interior points
\begin{equation}
\label{eq:part2a}
z \,\, =\,\, \lambda \cdot x(t_i) \,+\, (1-\lambda) \cdot x(t_j) \qquad {\rm where}\,\, \,\,0 < \lambda < 1 .
\end{equation}
The following $2 \times 2$ determinant is a polynomial in the parameter $\lambda$:
\begin{equation}
\label{eq:part2b}
f(\lambda) \quad =\quad {\rm det} \bigl( \,\phi(z) \,,\,x(t_i) - x(t_j)\, \bigr) . 
\end{equation}
We compute all real zeros of  the polynomial $f(\lambda)$ in the open interval $(0,1)$.
The zeros partition the  edge of $C$ into segments
where $\phi(z)$ points either inward or outward,
relative to  the convex region $C$. If there are no zeros
then the entire edge of $C$ is either inward pointing
or outward pointing.
In this manner we partition $\partial C$, and thereby solve (ii).
In order to compute the attainable region, we start new trajectories
from a sample of outward points.

\begin{figure}[h]
\begin{center}
 \includegraphics[width=0.6\textwidth]{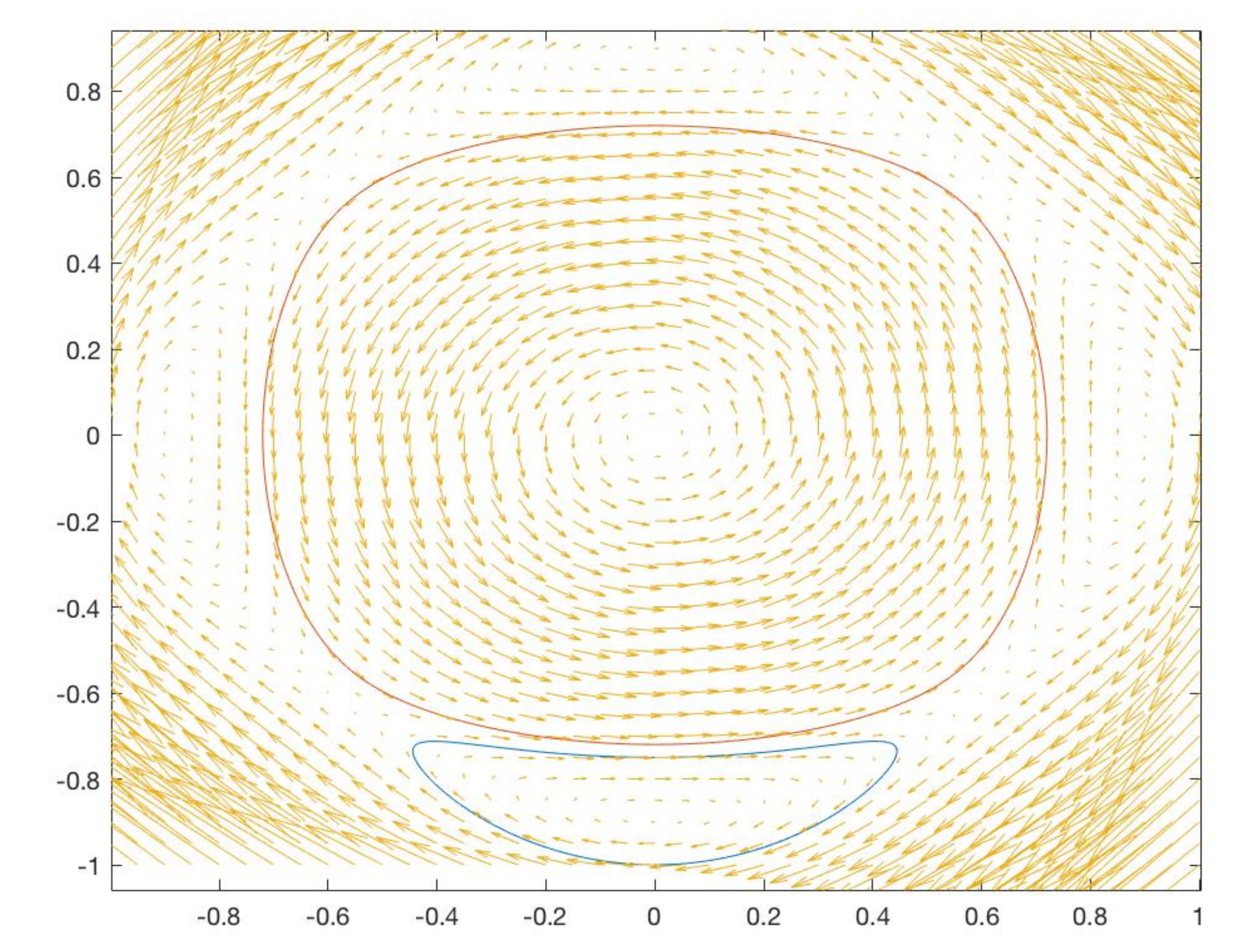}
 \caption{\label{fig:trott} The Hamiltonian vector field defined by the Trott curve
 and two of its trajectories.}
  \end{center}
 \end{figure}
 
 We tested our method on trajectories that are algebraic curves. 
Let $\mathcal{C}$ be the curve in $\RR^2$ defined by a polynomial equation $h(x,y) = 0$. The 
associated {\em Hamiltonian system} equals
\begin{equation}
\label{eq:hamiltonian}
 \dot x \, = \, \frac{\partial h}{\partial y} (x,y) \qquad {\rm and} \qquad
\dot y \, = \,- \,
\frac{\partial h}{\partial x} (x,y) . 
\end{equation}
Henceforth, we only consider Hamiltonian systems associated with polynomials. 
Thus $h$ is a polynomial in $x$ and $y$.
At any point that is not a critical point of $h$, the right hand side of (\ref{eq:hamiltonian})
is orthogonal to the  gradient vector of $h$. This means that the vector field is tangent to the
{\em level curves} $\,h(x,y) = c$, where $c $ ranges over $\RR$. From this we infer the following 
well-known result.

\begin{corollary} \label{cor:hamilton}
Every trajectory of (\ref{eq:hamiltonian}) is a piece of a level curve $\{h(x,y) = c\}$.
\end{corollary}

We close this section by
illustrating Hamiltonian systems and our solutions to Problems (i) and (ii)
for $n=2$, for a curve that is familiar in computational algebraic geometry.

\begin{example}[$n=2$] \label{ex:trott}
\rm The {\em Trott curve} is the quartic in the plane $\RR^2$ defined by
$$ h(x,y) \,\, = \,\,144 (x^4+y^4)- 225(x^2+y^2)+350x^2y^2+81. $$
This curve consists of four nonconvex ovals. Hence it has $28$ real bitangents.

The vector field for the Hamiltonian system (\ref{eq:hamiltonian}) is shown in
Figure \ref{fig:trott}, along with two of its trajectories.
Fix any point $(u,v)$ in $\RR^2$ and set $c = h(u,v)$.
Then the trajectory of (\ref{eq:hamiltonian}) that starts at $(u,v)$
travels on the quartic curve defined by $h(x,y) = c$.
Consider the starting point $ (0,-1)$, which lies on the original Trott curve $h(x,y) = 0$.
Its trajectory is one of the four ovals, namely the oval at the bottom that is
red in Figure \ref{fig:test2} (left) and blue in Figure \ref{fig:trott}.

\begin{figure}[h]
\begin{center}
 \includegraphics[width=.3\linewidth]{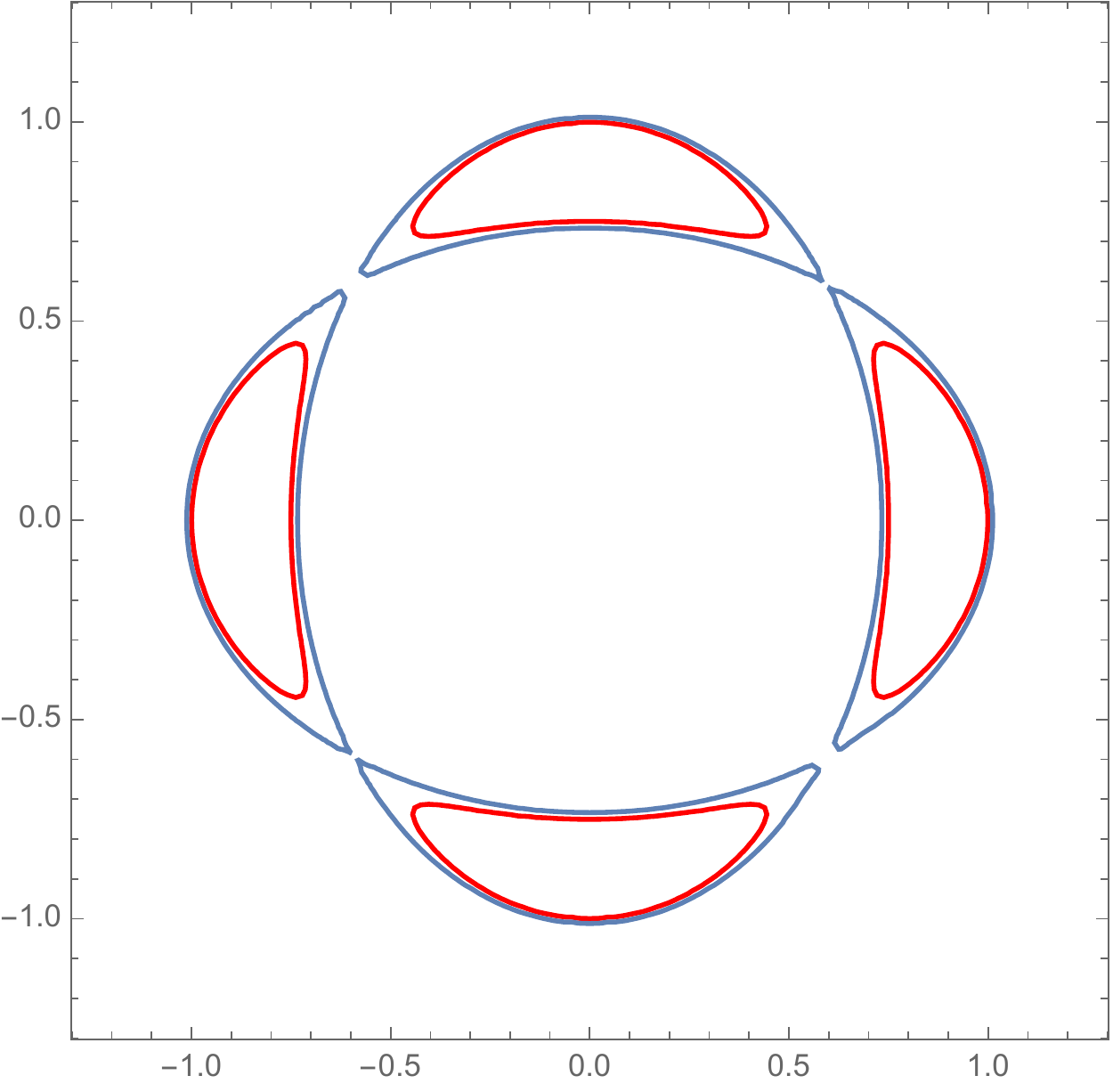} \qquad \qquad
 \includegraphics[width=.25\linewidth]{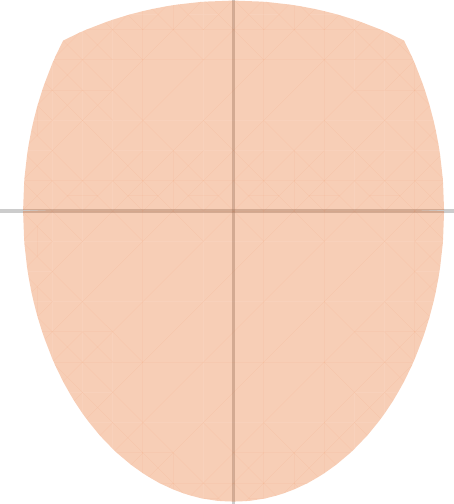}
\end{center} 
  \caption{A pair of ellipses encloses the Trott curve
  and bounds the attainable region.
  \label{fig:test2}}
\end{figure}

The region bounded by that oval is not convex.
Its convex hull is the convex trajectory. It has one 
bitangent edge, namely the segment
from $(-a,b)$ to $(a,b)$~where
$$ 
a =  0.4052937596229429488 \qquad {\rm and} \qquad
b = -0.7125251813139792270.
$$

This convex trajectory is not forward closed. 
This can be seen by taking any starting point $(x,b)$ where $0 < x < 0.239173943$.
 The resulting trajectory is a convex curve that lies above the
original oval. It is shown in red in Figure \ref{fig:trott}. 
The set of all trajectories
as $x$ ranges from $0$ to $a$ sweeps out the attainable region for $(0,-1)$.
This region is a semialgebraic set. Its boundary consists of parts of two ellipses.
Their union is the zero set of
$\,  h(x,y) - \frac{1053}{638} $.
The containment of the Trott curve in the ellipse is shown on the left in Figure~\ref{fig:test2}. The attainable region of
the lower oval in the Trott curve is the convex set of the right in Figure~\ref{fig:test2}. 
\end{example}

\section{Limiting Faces in Polyhedral Approximations}
\label{sec3}

We now study the limiting behavior of the facets of a sequence of polytopes. 
Each polytope is the convex hull 
of points sampled on a smooth, compact curve $\mathcal{C}$ in $\RR^n$ whose
convex hull $C = {\rm conv}(\mathcal{C})$ has dimension $n$. Our aim is to derive information about faces of
 $C$ from the facets of its polyhedral approximations. Even though there are lots of profound results on polytopal 
approximation of convex bodies, see e.g.\ \cite{Bro} for a survey, to the best of our knowledge there is
no study that gives results similar to the ones presented here. 
In addition to the approximation of $C$ in terms of Hausdorff distance, in the limit we also approach the  special facial structure of the convex hull $C$ of the curve.
Our approximation method, which is based on points sampled on that curve, accounts for this.  
We will elaborate on the structure of $\partial C$ in 
Section \ref{sec5}. Here we establish a theoretical basis for our algorithms. 

We wish to compute the boundary  of $C$ using a sequence
of inner approximations by convex polytopes, each obtained as the convex hull
of a path that approximates $\mathcal{C}$.
The {\em Hausdorff distance} of two compact sets $B_1$ and $B_2 $ in $ \RR^n$ is defined as
$$ d(B_1,B_2) \,\,=\,\,
\max\,\bigl\{\,\max_{x \in B_1} \min_{y \in B_2} \|x-y\| \,,\,\max_{y \in B_2} \min_{x \in B_1} \|x-y\| \,\bigr\}.$$
A sequence $\{B_{\nu}\}_{\nu \in \NN}$ of compact sets  is {\em Hausdorff convergent} to 
a fixed compact set $B$  if $\,d(B,B_\nu)\to 0\,$ for $\nu \to \infty$. 

A point $x$ of a compact convex set $B$ is {\em extremal} if it is not a proper convex combination of elements of $B$, that is, $\{x\}$ is a zero-dimensional face of $B$. Extremal points of a polytope
are called {\em vertices}. Even if $\{B_{\nu}\}_{\nu \in \NN}$ is a sequence of polytopes that Hausdorff converges to a polytope $B$, the limit of a convergent sequence of vertices $x_\nu$ of $B_\nu$ is not necessarily a vertex of $B$.
 For instance, consider $B_\nu = \conv(\{(0,0),(1,\frac{1}{\nu}),(2,0)\})$.  However, a converse holds.
 
\begin{lemma}\label{lem:0}
	Let $\{B_\nu\}_{\nu\in \NN} \rightarrow B$ be a Hausdorff convergent sequence of
	compact convex sets in $\RR^n$. For every extremal point $x$ of $B$ there exist
	extremal points $x_\nu$ of $B_\nu$ converging to $x$.
\end{lemma}

\begin{proof} Let $x$ be an extremal point of the limit body $B$. By Hausdorff convergence, there 
exists a sequence $\{x_\nu\}_{\nu \in \NN}$ with $x_\nu \in B_\nu$ that converges to $x$. 
By Carath\'eodory's Theorem, each $x_\nu$ is a convex combination of at 
 most $n+1$ extremal points $v_0^\nu,v_1^\nu,\dots,v_n^\nu$ of $B_\nu$. 
 
 Fix some $\varepsilon > 0$. Assume there is an infinite subset $N$ of $\NN$ such that
	$$\forall \nu \in N	\,\, \forall i \in \{0,1,\dots,n\}:\; \|x-v_i^\nu\| \,\geq \, \varepsilon.$$
By compactness, there is a subsequence $N'$ of $N$ such that, for each $i$,
the sequence $\{v_i^\nu\}_{\nu \in N'}$ converges to some $v_j$. 
Hence $x$ is a proper convex combination of $v_0,v_1,\dots,v_n \in B$. This contradicts 
$x$ being extremal in $B$. Therefore,
 for every $\varepsilon > 0$ there exists $\nu_0 \in \NN$ such that 
$$ \forall \nu \geq \nu_0 \,\; \exists i \in \{0,1,\dots,n\}:\; \|v_i^\nu - x\|\,< \, \varepsilon.$$
We thus obtain a sequence $\{x_\nu\}$ of extremal points 
$x_\nu$ of $B_\nu$ converging to $x$.
\end{proof}

An {\em $\varepsilon$-approximation} of the
given curve $\mathcal{C}$ is a finite subset $\mathcal{A}_\varepsilon \subset \mathcal{C}$ such that
$$ \forall y \in \mathcal{C}\,\, \exists x \in \mathcal{A}_{\varepsilon}: \|y-x\| \,\leq\, \varepsilon.$$
We consider a sequence $\{\mathcal{A}_{\varepsilon}\}_{\varepsilon {\scriptscriptstyle\searrow} 0}$ of $\varepsilon$-approximations, where $\varepsilon {\scriptscriptstyle\searrow} 0$ stands for a decreasing sequence $\{\varepsilon_\nu\}_{\nu \in \NN}$ of positive real numbers $\varepsilon_\nu$.
 The polytopes $A_\varepsilon=\conv(\mathcal{A}_{\varepsilon})$ can be described by their facets. Our goal is to study convergent sequences of facets $F_{\varepsilon}$ of $A_{\varepsilon}$ 
 in order to get information about the facial structure of the convex hull $C$ of the curve $\mathcal{C}$. 

\begin{proposition}\label{prop:1} Let $\{F_\varepsilon\}_{\varepsilon {\scriptscriptstyle\searrow} 0}$ be a Hausdorff convergent sequence of proper faces $F_\varepsilon$ of
the polytopes ${A}_\varepsilon$. Then its limit $F$ is contained in an exposed face of $C$.
\end{proposition}

\begin{proof}
We write the face $F_\varepsilon$ of the polytope $A_\varepsilon$ in the form
$\, F_\varepsilon = \{x \in {A}_\varepsilon \mid y_\varepsilon^T x = \gamma_\varepsilon\}$,
where  $\|y_\varepsilon\| = 1$ and
$y_\varepsilon^T x \leq \gamma_\varepsilon$ for  $x \in \mathcal{A}_\varepsilon$. 
 Since $\mathcal C$ is compact,
the sequence $\{\gamma_\varepsilon\}$ is bounded. Choose accumulation points $y$ and $\gamma$, respectively, of $\{y_\varepsilon\}_{\varepsilon>0}$ and $\{\gamma_\varepsilon\}_{\varepsilon>0}$. Since $y \neq 0$, $H:=\{x \in \RR^n\mid y^T x = \gamma\}$ is a hyperplane. By Lemma~\ref{lem:0}, any extremal point $x$ of $F$ is the limit of a sequence $\{x_\varepsilon\}_{\varepsilon {\scriptscriptstyle\searrow} 0}$ for $x_\varepsilon$ a vertex of $F_\varepsilon$. Every vertex of $F_\varepsilon$ belongs to $\mathcal{C}$. Thus $F$ is contained in $C \cap H$. 
It remains to show that $C$ is contained in the halfspace $H_-:=\{x \in \RR^n\mid y^T x \leq \gamma\}$. Assume there exists $x \in \mathcal{C}$ with $d:= y^T x - \gamma \,>\, 0$. Then there 
exists a sequence $\left\{  x_\varepsilon\right\}$
with $x_\varepsilon \in \mathcal{A}_\varepsilon$ converging to $x$ and such that $\{y_\varepsilon^T x_\varepsilon - \gamma_\varepsilon\}$ converges to $d$. This contradicts that the halfspace $\left\{ x \in \mathbb{R}^n  \mid y_\varepsilon^T x  \leq \gamma_\varepsilon \right\}$ contains $A_\varepsilon$.
\end{proof}

The limit $F$ in Proposition \ref{prop:1} may not be a face of $C$. This is shown in Figure~\ref{fig:img2} on
the left.
The following genericity assumptions on $\mathcal{C}$ will ensure that a Hausdorff convergent sequence of proper faces 
% of facets   %BS??
$F_\varepsilon$ of the polytopes $A_\varepsilon$ converges to a proper face $F$ of  the body~$C$:
\begin{itemize}
	\item[(H1)] Every point on the curve $\mathcal{C}$ that is in the boundary of $C$ is an extremal point of $C$. 
%\vspace{-0.1in}	
	\item[(H2)] Every polytope face of $C$ is a simplex.
%\vspace{-0.1in}	
	\item[(H3)] Intersecting the curve $\mathcal{C}$ with a hyperplane always results in a finite set.
\end{itemize}

\begin{figure}[h]
	\begin{center}
	\begin{tikzpicture}
	\filldraw [gray!20]  (-2,-1)-- (0,-1.25)-- (2,-1) -- (2.2,0)-- (1.6,0.75) --  (0,1.03) --(-1.6,0.75) --(-2.2,0)--cycle;
\draw [red, thick] plot [smooth cycle,tension=.8] coordinates {(-2,-1) (0,-1.25) (2,-1)  (2.2,0) (1.6,0.75)  (1,1) (0.5,0.25) (0,1.03) (-0.5,.25)  (-1,1) (-1.6,0.75) (-2.2,0)};
\draw [blue, ultra thick,line cap=round]  (1.15,1.03) -- (0,1.03) -- (-1.15,1.03);
\filldraw  (-2,-1) circle (2pt);
\filldraw  (0,-1.25) circle (2pt);
\filldraw  (2,-1)circle (2pt);
\filldraw  (2.2,0) circle (2pt);
\filldraw  (1.6,0.75)circle (2pt);
\filldraw  (0,1.03)  circle (2pt);
\filldraw  (-1.6,0.75) circle (2pt);
\filldraw  (-2.2,0) circle (2pt);
\filldraw  (0.5,.25) circle (2pt);
\filldraw  (-0.5,.25) circle (2pt);
\draw [above] (0.2,1.1) node {$F$};
\draw [below] (0,1) node {$y$};
\draw [left] (-2.4,-0.5) node {$\mathcal{C}$};
\draw [left] (1,-0.5) node {$A_\varepsilon$};

%\filldraw [gray!20]  (-4,-2)-- (0,-2.5)-- (4,-2) -- (4.4,0)-- (3.2,1.5) --  (0,2.07) --(-3.2,1.5) --(-4.4,0)--cycle;
%\draw [red, thick] plot [smooth cycle,tension=.8] coordinates {(-4,-2) (0,-2.5) (4,-2)  (4.4,0) (3.2,1.5)  (2,2) (1,.5) (0,2.07) (-1,.5)  (-2,2) (-3.2,1.5) (-4.4,0)};
%\draw [blue, ultra thick,line cap=round]  (2.29,2.07) -- (0,2.07) -- (-2.29,2.07);
%\filldraw  (-4,-2) circle (2pt);
%\filldraw  (0,-2.5) circle (2pt);
%\filldraw  (4,-2)circle (2pt);
%\filldraw  (4.4,0) circle (2pt);
%\filldraw  (3.2,1.5)circle (2pt);
%\filldraw  (0,2.07)  circle (2pt);
%\filldraw  (-3.2,1.5) circle (2pt);
%\filldraw  (-4.4,0) circle (2pt);
%\filldraw  (1,.5) circle (2pt);
%\filldraw  (-1,.5) circle (2pt);
%\draw [above] (0.4,2.2) node {$F$};
%\draw [below] (0,2) node {$y$};
%\draw [left] (-4.8,-1) node {$\mathcal{C}$};
%\draw [left] (2,-1) node {$A_\varepsilon$};
\end{tikzpicture} \qquad \qquad
	\begin{tikzpicture}
	\filldraw [gray!20]  (0,1) -- (-0.75,-0.5) -- (-1,-1.1)--(-0.5,-1.5) --(0,-1.25)--cycle;
\draw [red, thick] plot [smooth ,tension=.8] coordinates {(0,1) (-.15,0.5) (-.1,-0.25) (-.75,-0.5) (-1,-1.1) (-.5,-1.5) (0,-1.25) (-.35,-0.65) };
%\draw [blue, ultra thick]  (2.29,2.07) -- (0,2.07) -- (-2.29,2.07);
\filldraw  [blue] (0,1)circle (2pt);
\filldraw  (-0.75,-0.5)circle (2pt);
\filldraw  (-1,-1.1) circle (2pt);
\filldraw  (-0.5,-1.5)circle (2pt);
\filldraw  (0,-1.25)circle (2pt);
\filldraw  (-.15,0.5)  circle (2pt);
\filldraw  (-.1,-0.25)circle (2pt);
\filldraw  (-.35,-0.65) circle (2pt);
%\filldraw  (1,.5) circle (2pt);
%\filldraw  (-1,.5) circle (2pt);
%\draw [above] (0.4,2.2) node {$F$};
\draw [right] (0,1) node {$F$};
\draw [left] (-1.15,-0.5) node {$\mathcal{C}$};
\draw [left] (-0.25,-1) node {$A_\varepsilon$};
%\filldraw [gray!20]  (0,2) -- (-1.5,-1) -- (-2,-2.2)--(-1,-3) --(0,-2.5)--cycle;
%\draw [red, thick] plot [smooth ,tension=.8] coordinates {(0,2) (-.3,1) (-.2,-0.5) (-1.5,-1) (-2,-2.2) (-1,-3) (0,-2.5) (-.7,-1.3) };
%%\draw [blue, ultra thick]  (2.29,2.07) -- (0,2.07) -- (-2.29,2.07);
%\filldraw  [blue] (0,2)circle (2pt);
%\filldraw  (-1.5,-1)circle (2pt);
%\filldraw  (-2,-2.2) circle (2pt);
%\filldraw  (-1,-3)circle (2pt);
%\filldraw  (0,-2.5)circle (2pt);
%\filldraw  (-.3,1)  circle (2pt);
%\filldraw  (-.2,-0.5) circle (2pt);
%\filldraw  (-.7,-1.3) circle (2pt);
%%\filldraw  (1,.5) circle (2pt);
%%\filldraw  (-1,.5) circle (2pt);
%%\draw [above] (0.4,2.2) node {$F$};
%\draw [right] (0,2) node {$F$};
%\draw [left] (-2.3,-1) node {$\mathcal{C}$};
%\draw [left] (-0.5,-2) node {$A_\varepsilon$};
\end{tikzpicture} \quad
	\end{center}
\caption{A Hausdorff convergent sequence of facets $F_\varepsilon$ of $A_\varepsilon$ 
need  not  converge to a face $F$ of $C$. The face $F$ on the left contains a curve 
point $y \in \mathcal{C}$ which is not extremal in $C$. 
The endpoint $F$ of the curve $\mathcal{C}$ on the right is an exposed face of $C$ but it is not uniquely exposed.
There is no sequence of facets $F_\varepsilon$ of $A_\varepsilon$ that Hausdorff converges to that face~$F$.
}
\label{fig:img2}
\end{figure}
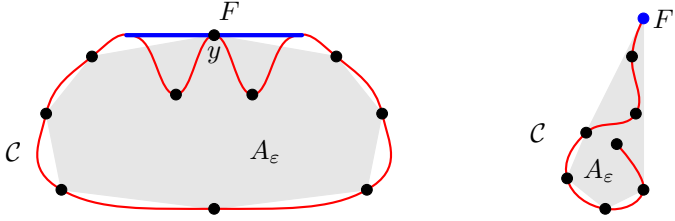

We now give a sufficient condition that proper faces of $C$ are polytopes.

\begin{proposition}\label{prop:2}
If $\,\mathcal{C}$ satisfies  {\rm (H3)}, then every proper face of $\,C$ is a polytope.
\end{proposition}

\begin{proof}
	A proper face $F$ of $C$ belongs to some hyperplane. By (H3),
	 the set $\mathcal{C} \cap F$ is finite.
	  Since $F$ is a face of $C$, an extremal point of $F$ is also an extremal point of $C$.  
	  All extremal points of $C$ belong to $\mathcal{C}$, since they cannot be expressed as a proper convex combination of curve points. Thus, $F$ is a polytope as it has only finitely many extremal points.
\end{proof}

\begin{proposition}\label{prop:3} 
Suppose that {\rm (H1)}, {\rm (H2)} and {\rm (H3)} hold. 
Let $\{F_\varepsilon\}_{\varepsilon {\scriptscriptstyle\searrow} 0}$ be a Hausdorff convergent sequence of proper faces $F_\varepsilon$ of ${A}_\varepsilon$. Then its limit $F$ is a proper face of~$C$.
\end{proposition}

\begin{proof} By Proposition \ref{prop:1}, $F$ is contained in an exposed face $G$ of $C$, in particular, in a proper face of $C$. Let $G$ be the smallest face of $C$ containing $F$. By assumption (H3) and Proposition \ref{prop:2}, $G$ is a polytope. The extremal points of $F$ belong to $\mathcal{C}$. The subset $F$ of $G$ is in the boundary of $C$. By assumption (H1), the extremal points of $F$ are extremal points of 
$C$. Since $F \subset G \subset C$, they are also extremal points of $G$. Thus either $F=G$ or $F$ is a sub-simplex of $G$ by assumption (H2). The latter case contradicts the minimality of $G$.
\end{proof}

Let $F$ be a proper face of $C$. We seek a sequence $F_\varepsilon$ of facets of $A_\varepsilon$ Hausdorff converging to $F$. In general, such a sequence does not exist, even under the assumptions (H1), (H2), (H3). 
We need to additionally require the face $F$
to be {\em uniquely exposed}, that is, there is a unique
 halfspace $H^+$ with $C \subset H^+$ and $F = C \cap - H^+$. For an
 example see Figure~\ref{fig:img2} (right).

\begin{theorem} \label{th:1} Assume {\rm (H1)} and let $F$ be a simplex which is a uniquely exposed face of $\,C$. Then $F$ is the Hausdorff limit of a sequence $\{F_\varepsilon\}_{\varepsilon {\scriptscriptstyle\searrow}0}$ of facets of ${A}_\varepsilon$.
\end{theorem}

\begin{proof} Let $v_0,\ldots,v_k$ be the vertices of $F$. Without loss of generality, $\frac{1}{k+1}\sum_{i=0}^k v_i = 0$. The halfspace which defines $F$ in $C$ has the form
	$ H^+_\gamma :=\{x \in \RR^n\mid h^T x \geq \gamma\}$,
	where $\gamma = 0$ and $h\in \RR^n$ with $\|h\|=1$ is unique.
	We claim that for every $\delta > 0$ there exists $\gamma >0$ such that 
\begin{equation}\label{eq:iso}
	\forall\, y \in \mathcal{C} \backslash H_\gamma^+ \,\,\, \exists \,i \in \{0,1,\dots,k\}:\;
	 \|y-v_i\| \,< \,\delta.
\end{equation}
We prove this by contradiction. Suppose
there exists  $\delta > 0$ such that for all $\gamma > 0$ there exists $y \in \mathcal{C} 
\backslash H_\gamma^+$  with $\|y-v_i\| \geq \delta$ for all vertices $v_i$ of $F$. Thus we can construct a sequence of curve points approaching $-H^+ = -H^+_0$ but maintain a distance of at least $\delta$ 
from each $v_i$. By compactness of $\mathcal{C}$, this sequence can be assumed to 
converge to some $z \in -H^+ \cap \mathcal{C} \subset F$. Since the curve point $z$ belongs to the boundary of $C$, assumption (H1) ensures that $z$ is an extremal point of $C$. Hence, $z$ is a vertex of $F$ different from $v_0,\dots,v_k$, a contradiction.

For $\varepsilon > 0$, we consider the linear program
\begin{equation}\label{eq:lp}
	\min \mu \quad \text{s.t.}\quad \mu h \in A_{\varepsilon}.
\end{equation}
We claim that the following holds for sufficiently small $\varepsilon>0$:
\begin{equation}\label{eq:int}
\text{span}\,h \,\cap \,\text{int}\,A_\varepsilon\,\, \neq \,\,\emptyset.
\end{equation}
 Assume the contrary. Then, for all $\varepsilon > 0$, $\text{span}\,h$ and ${A}_\varepsilon$ can be separated weakly by a hyperplane $H(\varepsilon) = \{x \in \RR^n\mid h_\varepsilon^T x = \gamma_\varepsilon\}$ with $\|h_\varepsilon\|=1$. By a compactness argument, a subsequence of $\{(h_\varepsilon,\gamma_\varepsilon)\}_{\varepsilon {\scriptscriptstyle\searrow} 0}$ converges to $(\bar h, \bar \gamma)$ with $\|\bar h\|=1$.
  The hyperplane $\bar H = \{x \in \RR^n\mid \bar h^T x = \bar \gamma\}$ weakly separates $\text{span}\,h$ and $C$. Since $0$ is contained in both $C$ and $\text{span}\,h$, we have $\bar \gamma = 0$. Since $0$ is a relative interior point of $F$, we must have $F \subset \bar H$. Hence, $F$ is exposed with respect to a halfspace $\bar H^+ := \{x \in \RR^n\mid \bar h^T x \geq \bar \gamma\}$ corresponding to $\bar H$. Since $H^+ \neq \bar H^+$, this contradicts the assumption that $F$ is uniquely exposed with respect to $H^+$.

Let $\mu_\varepsilon$ be the optimal value of \eqref{eq:lp}. We have
 $\{\mu_\varepsilon\}_{\varepsilon{\scriptscriptstyle\searrow} 0}= 0$.
  We will use linear programming duality to show that, for
   $\varepsilon >0$ sufficiently small, $\mu_\varepsilon h$ belongs to a facet    of the form
$ \,F_\varepsilon \,=\, {A}_\varepsilon\, \cap \,\{x \in \RR^n\mid y_\varepsilon^T x = \mu_\varepsilon\}\,$
where $\,h^T y_\varepsilon = 1$.
Let $M_\varepsilon$ is the matrix with columns $\mathcal{A}_\varepsilon$ 
and $e=(1,\dots,1)^T$.
 The linear program dual to \eqref{eq:lp} is
\begin{equation}\label{eq:dual_lp}
	\max \,\eta \quad \text{s.t.}\quad M_\varepsilon^T y - e \eta \geq 0
	\,\,\,{\rm and} \,\,\, h^T y = 1,
\end{equation}
Let $(y_\varepsilon,\eta_\varepsilon)$ denote an optimal solution of \eqref{eq:dual_lp}. By duality, $\mu_\varepsilon=\eta_\varepsilon$. We conclude that
the set $F_\varepsilon$ is a face of the polytope $A_\varepsilon$.

 To see that $F_\varepsilon$ is a facet of $A_\varepsilon$, we replace \eqref{eq:lp} and \eqref{eq:dual_lp} by the pair of dual problems
\begin{equation}\label{eq:lp_hom}
	\min 0 \mu \quad \text{s.t.}\quad \mu h \in A_{\varepsilon},
\end{equation}
\begin{equation}\label{eq:dual_lp_hom}
	\max \eta \quad \text{s.t.}\quad M_\varepsilon^T y - e \eta \geq 0 \,\,\, {\rm and} \,\,\, h^T y = 0.
\end{equation}
Using complementary slackness, we conclude from \eqref{eq:int}
 that $(y,\eta)=(0,0)$ is the unique optimal solution of \eqref{eq:dual_lp_hom}. Hence
 the set of optimal solutions of \eqref{eq:dual_lp} is bounded, and we can choose
  $(y_\varepsilon,\eta_\varepsilon)$  to be a vertex. At least $n$ linearly independent inequalities in \eqref{eq:dual_lp} hold with equality at
 $(y,\eta)=(y_\varepsilon,\eta_\varepsilon)$. These correspond to $n$ affinely independent 
 points in~$\mathcal{A}_\varepsilon$, all belonging to the hyperplane $\{x \in \RR^n\mid y_\varepsilon^T x = \mu_\varepsilon\}$. This shows that $F_\varepsilon$ is a facet of ${A}_\varepsilon$.

From \eqref{eq:iso}, we conclude that, for sufficiently small $\varepsilon > 0$,
 the point $\mu_\varepsilon h \in F_\varepsilon$ (which approaches the mean of the vertices of $F$) can be represented only by elements $y \in \mathcal{A}_\varepsilon$ with $\|y-v\|<\delta$ for some vertex $v$ of $F$. Since $F$ is a simplex, each vertex $v$ of $F$ is used in this representation. 
Hence, for each vertex $v$ of $F$ there exists a vertex $y$ of $F_\varepsilon$ with $\|y-v\|<\delta$.

We claim that, if $\varepsilon > 0$ is sufficiently small 
then for every vertex $y$ of $F_\varepsilon$ there exists a vertex $v$ of $F$ with $\|y-v\|<\delta$. Assume the contrary. Then, by \eqref{eq:iso}, for any small $\varepsilon > 0$,
there is a vertex $y_\varepsilon$ of $F_\varepsilon$ with $y_\varepsilon \in \mathcal{C} \cap H^+_\gamma$.
  By compactness of $\mathcal{C}$, we may assume that $\{y_\varepsilon\}_{\varepsilon{\scriptscriptstyle\searrow} 0}$ 
    converges to some $\bar y \in \mathcal{C} \cap H^+_\gamma$.
By Proposition \ref{prop:1}, $\conv \left(F \cup \{\bar y\}\right)$ belongs to an exposed face of $C$. Since $\bar y \not\in -H^+$, this contradicts the assumption that $F$ is uniquely exposed.
We conclude that for every $\delta > 0$ we find $\varepsilon_0 > 0$ such that $d(F_\varepsilon,F)<\delta$ for all $0 < \varepsilon \leq \varepsilon_0$.
Hence, the simplex face $F$ of $C$ is the Hausdorff limit of the sequence
$\{F_\varepsilon\}_{\varepsilon {\scriptscriptstyle\searrow}0}$, as desired.
\end{proof}

\begin{remark} \rm  In the proof of Theorem \ref{th:1}, we  construct a sequence $\{F_\varepsilon\}_{\varepsilon {\scriptscriptstyle\searrow}0}$ of facets of ${A}_\varepsilon$ 
whose vertices converge to the 
 vertices of $F$. This is stronger than Hausdorff convergence.
 \end{remark}

\begin{remark} \rm 
In our proofs we had assumed, for simplicity, that the curve $\mathcal{C}$
is smooth and that the sampled points lie exactly on $\mathcal{C}$.
Our results should extend to nonsmooth curves and
to points that are sampled around $\mathcal{C}$ with a given accuracy. The investigation of polyhedral approximations  of $\,C = {\rm conv}(\mathcal{C})\,$ under
such weaker assumptions is left for future work.
\end{remark}

\section{Convex Hulls in Bensolve}
\label{sec4}

The key step in our solution to Problem (i) is the computation 
of the convex hull of an $\varepsilon$-approximation of a curve $\mathcal{C}$.
There are many methods and implementations for  convex~hulls.
For this paper, the software {\tt Bensolve Tools} ~\cite{CLW}
was used. It is based on Benson's algorithm; see e.g.~\cite{ELS}.
One reason for that choice
 is the output sensitivity of the underlying method. This means that the runtime is mainly dependent on the number of facets and vertices of the polytope that is the convex hull  
and its dimension. In particular,
the number of sampled  points in the interior of $C_\mathcal{A}$ only marginally influences the computation time. Another advantage of Benson's algorithm is the possibility to set the  parameter $\varepsilon$ in Algorithm \ref{alg:bensons}. This feature enables the approximative representation of highly complex convex hulls in a reasonable amount of time. In addition, the process can be aborted at any point while still providing an outer approximation. For small values of $\varepsilon$, we obtain exact solutions, up to numerical inaccuracy of the vertex enumeration routine.

We next discuss this software, its underlying methodology, and how we apply~it.
{\tt Bensolve} \cite{LW} is a solver for multiple objective linear programs (MOLP).
In {\tt Bensolve Tools} it is utilized to perform many polyhedral calculus operations, 
among them convex hull.
The key insight behind this is that multiple objective linear programming is equivalent to polyhedral projection \cite{LW2}. Convex hull computation is a special case of polyhedral projection. 
This follows from standard facts in Ziegler's textbook
\cite[Chapter~1]{Zie}. We state it as follows:

\begin{lemma}
The convex hull of  $\,\mathcal{V} = \{v^1,v^2,\ldots,v^k\}  $ in $ \RR^n$ is the polytope
$$
 {\rm conv}(\mathcal{V}) \,\,=\,\, \{ y \in \RR^n \,\mid \,\exists \lambda \in \RR^k:\; \lambda \geq 0,\; e^T \lambda = 1,\;y = V\lambda\},
$$ 
where $ V\in \RR^{n \times k}$ is the matrix with set of columns $\mathcal{V}$ and
$e=(1,1,\ldots,1)^T$.
Hence, the convex hull of $\,\mathcal{V}$ is a projection into $\RR^n$ of the polytope 
$$
 Q\,\,=\,\,\{ (y,\lambda) \in \RR^n \times \RR^k \,\mid \, \lambda \geq 0,\; e^T \lambda = 1,\;y = V\lambda\}.
$$
\end{lemma}

To understand the computation of  ${\rm conv} (\mathcal{V})$ from $Q$,
let us turn to an arbitrary polyhedral projection problem. By Fourier-Motzkin Elimination, 
every linear projection of a polyhedron is a polyhedron. This leads to the concept of a {\em P-representation} 
of a polyhedron. Let $M \in \RR^{n\times k}$, $B \in \RR^{m\times k}$, $a \in \RR^m$ be given.
The triple $(M,B,a)$ represents the polyhedron
\begin{equation}\label{pp}
	P \,\,=\,\, \{Mx \mid Bx \geq a\} \,\,= \,\,\{y \in \RR^n \mid \exists x \in \RR^k:\; y=Mx,\; Bx \geq a\} .
\end{equation}

In what follows, we restrict to polytopes (bounded polyhedra). 
Given a P-representation (\ref{pp}) of a polytope, the polyhedral projection problem is to compute an irredundant  {\em V-representation}, i.e. a representation as convex hull of finitely many points, and an irredundant {\em H-representation}, i.e. a representation by finitely many linear inequalities (cf.~\cite{Zie}).

Given a triple $(M,B,a)$ as above, the associated 
{\em multiple objective linear program}~is
\begin{equation}\label{molp}
	\min\, Mx \quad \text{s.t.} \quad B x \geq a. \tag{MOLP}
\end{equation}
The {\em upper image} of the program (\ref{molp})  is the polyhedron
\begin{equation}\label{upperimg}
	\mathcal{P} \,\,=\,\, \bigl\{\, y \in \RR^n \mid \exists x \in \RR^k:\; y \geq Mx ,\; Bx \geq a \,\bigr\}.
\end{equation}
A solution of (\ref{molp}) consists of irredundant V- and H-representations of  $\mathcal{P}$.
This concept of solution can be used to address the polyhedral projection problem:

\begin{proposition}[cf. {\cite[Theorem 3]{LW2}}]
The solution of the MOLP
\begin{equation}\label{molp1}
	\min \begin{pmatrix} M \\ -e^T M\end{pmatrix} x \quad \text{s.t.} \quad B x \geq a,
\end{equation}
yields an irredundant V- and H-representation of the P-represented polyhedron (\ref{pp}).
\end{proposition}

The upper image of the MOLP in (\ref{molp1}) is the polyhedron
\begin{equation}\label{upperimg1}
\bar{\mathcal{P}} \,\,= \,\,\bigl\{ \,(y,z) \in \RR^n \times \RR \,\mid \, y \geq Mx,\; z \geq -e^T M x,\; 
	Bx \geq a \,\bigr\}.
\end{equation}

\begin{corollary}\label{cor:derive_rep}
The polytope $P = \{Mx \mid Bx \geq a\}$ is obtained from $\bar{\mathcal{P}}$ by setting
$$ P\,\, =\,\, \{ y \in \RR^n \mid \exists z:\; (y,z) \in \bar{\mathcal{P}}, \; e^T y + z = 0\}.$$
An irredundant V-representation of $P$ derives from the set of vertices of $\bar{\mathcal{P}}$ by deleting their last components. An H-representation of $\bar{\mathcal{P}}$ gives an  H-representation of $P$ by adding the equation $z=-e^T y$.
\end{corollary}

{\tt Bensolve} computes V- and H-representations of the upper image (\ref{upperimg1}) using
Algorithm \ref{alg:bensons}. This is a version of {\em Benson's algorithm}. It applies to upper images satisfying 
$\mathcal{P} \subseteq y+\mathbb{R}^{n}_{\geq 0}$ for some $y \in \mathbb{R}^n$. This version suffices for handling projections of polytopes including the representation of the convex hull of finitely many points. 
Since the algorithm is numerical, we work with a prescribed tolerance $\varepsilon > 0$.
The output is an {\em $\varepsilon$-approximation} to the
upper hull $\mathcal{P}$, i.e.~it is a polyhedron $\mathcal{O}$
that is {\em $\varepsilon$-close} to $\mathcal{P}$ in the sense that
$ \,\varepsilon e+ \mathcal{O} \subseteq \mathcal{P} \subseteq \mathcal{O}$.

\begin{algorithm_thm} (Benson's algorithm) 
\label{alg:bensons}

\begin{algorithm}[H]
\SetKwData{CommonVerts}{CommonVerts}
\SetKwFunction{Conv}{Conv} \SetKwFunction{Graph}{Graph} 
\SetKwFunction{Ben}{Bensolve} \SetKwFunction{AggregateShort}{AggregateShort}
\SetKwFunction{ConnectedComponents}{ConnectedComponents}
\SetKwInOut{Input}{input}\SetKwInOut{Output}{output}
\Input{(MOLP) given  by the matrices $M$, $B$ and vector $a$; a tolerance  $\varepsilon \geq 0$. }
\Output{$\varepsilon$-close V-representation $\mathcal{V}$ and H-representation $\mathcal{H}$ 
of $\mathcal{P}$ in (\ref{upperimg}).
}
\BlankLine
$T \leftarrow \emptyset$\\
Compute the H-representation $\mathcal{H}$ of an outer approximation of $\mathcal{P}$ having 
 the same recession cone as $\mathcal{P}$. Compute the corresponding V-representation~$\mathcal{V}$. \\ 
\While{$( \mathcal{V}\setminus T) \neq \emptyset$
}{
Choose a vertex $v \in \mathcal{V}\setminus T $. \\
Compute the solution $t^*$ of the linear program $\min \left\{ t \mid v+te \in \mathcal{P}  \right\}$. \\
Compute the solution $(u^*,w^*)$ of the dual linear program 
$\max \left\{ a^Tu -v^Tw \mid B^Tu=M^Tw,\; e^Tw=1,\; w\geq 0 ,\;u \geq 0 \right\}$. \\
\eIf{$t^* \geq \varepsilon$
}{
Refine  $\mathcal{H}$ by adding $ \left\{y \mid (w^*)^Ty \geq a^Tu^* \right\}$ to the description. \\
Update $\mathcal{V}$ by performing vertex enumeration on $\mathcal{H}$.
}
{$T \leftarrow T \cup \{ v \}$}
}
\end{algorithm}
\end{algorithm_thm}

One starts with an initial outer polyhedral approximation of $\mathcal{P}$. Both
 an H-representation and a V-representation are stored.
 Until the tolerance $\varepsilon$ is reached, each iteration adds a linear inequality to refine the outer 
 approximation of $\mathcal{P}$. An iteration step starts by choosing a vertex $v$ of the current  polyhedron. 
 The V-representation  is updated after adding an inequality. From $v$ one moves in direction 
 $e=(1,\ldots,1)^T$ to the boundary point $y=v+t^*e$ of $\mathcal{P}$. To this end, a linear program has to be solved. The solution of the dual linear program yields the desired linear inequality which cuts off $v$ and 
 holds with equality in $y$. 
Algorithm~\ref{alg:bensons} terminates and computes both V- and H-representation 
of an $\varepsilon$-approximation of $\mathcal{P}$.
For computations in this paper we use the {\em dual Benson algorithm} \cite{ELS}. 
It is dual to Benson's algorithm and provides an inner approximation for the upper hull~$\mathcal{P}$.  

We employ {\tt Bensolve} for obtaining a polyhedral approximation of the convex hull of 
a smooth curve $\mathcal{C}$ in $\RR^n$. This is done by computing the
convex hull of a sufficiently large finite subset $\mathcal{A}$ of $\mathcal{C}$. 
The output gives both an irredundant H- and V-representation of an inner $\varepsilon$-approximation $\mathcal{C}_\mathcal{A}$ of $\conv (\mathcal{A})$, and this is our approximation to ${\rm conv}(\mathcal{C})$.
 All facets and all vertices of $\mathcal{C}_\mathcal{A}$  are known after such a computation. 
 The output also contains the {\em incidence matrix} $I_\mathcal{A}$ for facets and vertices 
of $\mathcal{C}_\mathcal{A}$
and the {\em adjacency matrix} $A_\mathcal{A}$ for vertices of $\mathcal{C}_\mathcal{A}$.

\begin{example}\label{ex:yellow_green_1} \rm
Let $\mathcal{C}$ be the trigonometric space curve parametrically given by
\begin{equation} \label{eq:yg1}
\theta \, \mapsto \,\left( \cos(\theta),\sin(2 \theta),\cos(3 \theta) \right) .
\end{equation}
Its convex hull 
$C = {\rm conv}(\mathcal{C})$ is  shown in \cite[Figure 1]{RS2}.
We select the sample points 
\begin{equation}
\label{eq:bytakingA}
\begin{matrix}
\mathcal{A} \,\,= \,\, \left\{ \left( \cos\left(\frac{2k\pi}{N}\right),\sin\left(\frac{4k\pi}{N}\right),\cos\left(\frac{6k\pi}{N}\right) \right) \; \, \middle| \, \; k =  0,\ldots,N-1   \right\}.
\end{matrix}
\end{equation}
Using {\tt Bensolve}, as described above, 
we can compute irredundant V- and H-representations of
the inner approximation $C_\mathcal{A}$ of the polytope $\conv\left(\mathcal{A}\right)$
for various values of $N$ and with specified accuracy $\varepsilon$.
For instance, let $N = 100$ and $\varepsilon=10^{-9}$.
The sample (\ref{eq:bytakingA}) is shown on the left in Figure~\ref{fig:yellow_green_1}.
Its convex hull is the polytope $C_\mathcal{A}$ on the right in Figure~\ref{fig:yellow_green_1}.
It has $70$ vertices and $102$ facets, so, by Euler's relation, it has $170$ edges.
Thus the incidence matrix $I_\mathcal{A}$ is of size $102 \times 70$ and has $340$ nonzero entries. 
The $70 \times 70$ adjacency matrix $A_\mathcal{A}$ also has $340$ nonzero entries.
The polytope $C_\mathcal{A}$ in Figure~\ref{fig:yellow_green_1} already
looks like \cite[Figure 1]{RS2} and Figure \ref{fig:yellowgreen}.
The picture of $C_\mathcal{A}$ reveals the edge surfaces of $C$ and the two triangles in $\partial C$.
The identification of such patches from the {\tt Bensolve} output is our theme in 
Section~\ref{sec5}.
\end{example}

\begin{figure}[htp]
\includegraphics[width=4.2cm]{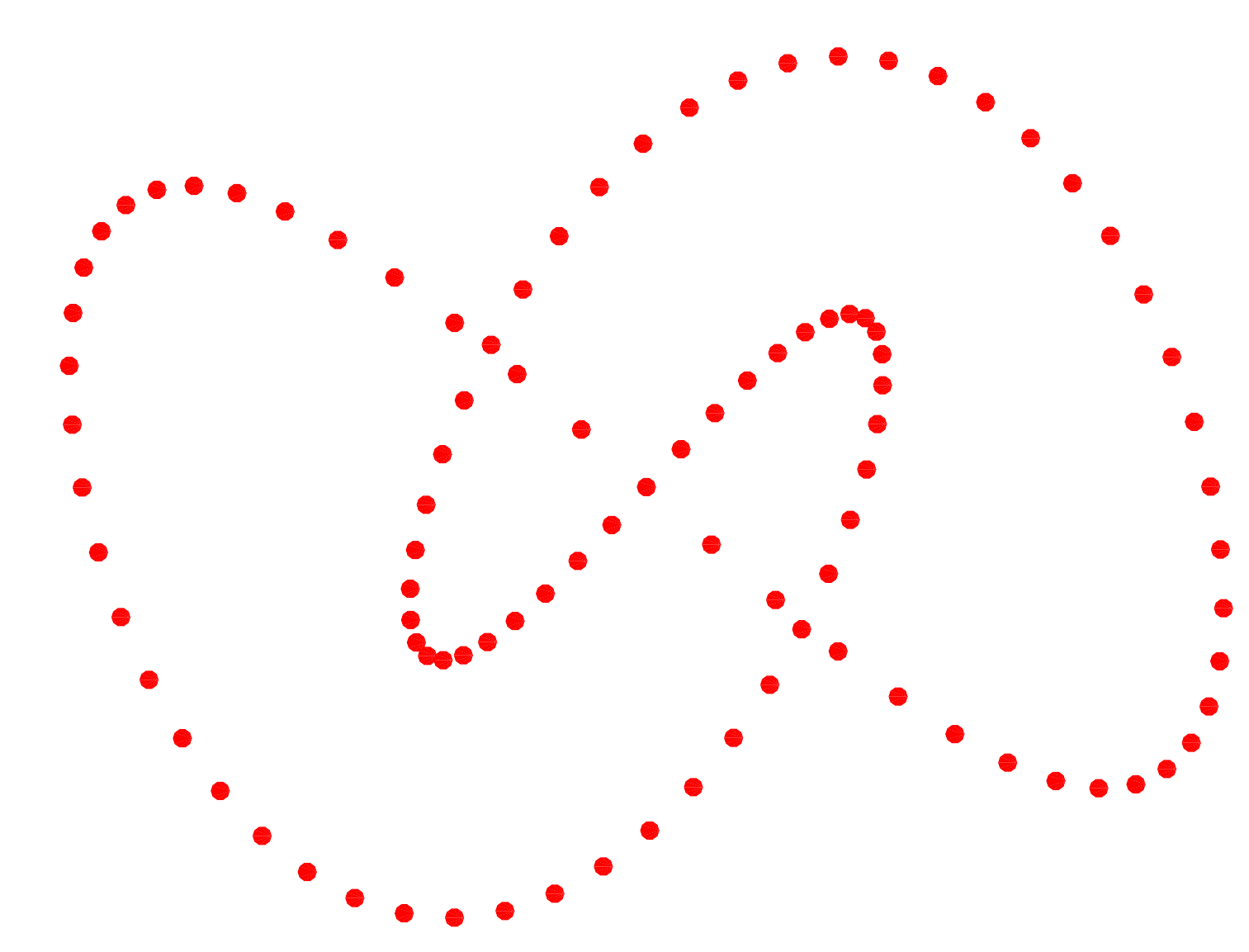} \qquad \qquad
\includegraphics[width=5cm]{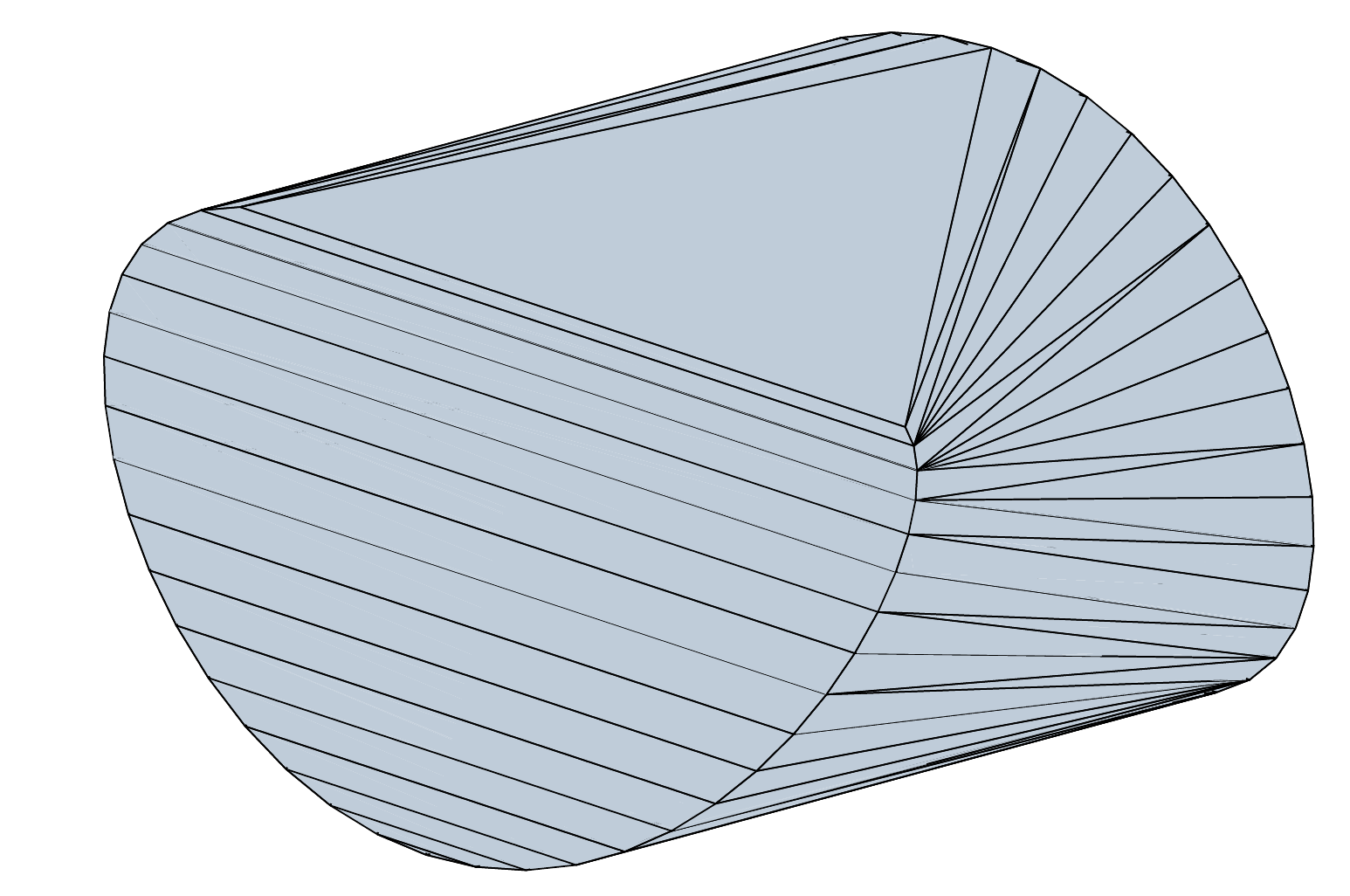}
\caption{\label{fig:yellow_green_1}
A sample of points (left) from a space curve and its convex hull (right).}
\end{figure}

\section{Detection of Patches} \label{sec5}

Let $C$ be a full-dimensional compact convex body in $\RR^n$. The boundary of $\partial C$ 
is an $(n-1)$-dimensional set whose subset  $\partial C_{sm}$ of smooth points is dense.
We shall stratify $\partial C_{sm}$ into finitely many manifolds we call {\em patches}.
Each patch is an $(n-k-1)$-dimensional family of $k$-faces of $C$.
For a typical convex body of dimension $n=3$, the boundary is comprised of
surfaces of extreme points ($k=0$), curves of edges ($k=1$),
and finitely many  facets ($k=2$).
For the general definition, we use the concept of the normal cycle of a convex body. 
Let $\mathbb{S}^{n-1}$ denote the unit $(n - 1)$-sphere. 
Following \cite[eqn (10)]{Fu}, the {\em normal cycle} of $C$ equals
$$ N(C) \quad = \quad \bigl\{\,(u,v) \in \RR^n \times \mathbb{S}^{n-1}\,
:\, v \cdot (u-u') \geq 0 \,\, \hbox{for all}\,\, u' \in C \,\bigr\}. $$

If $\partial C$ is smooth then $N(C)$ is a Legendrian submanifold of dimension $n-1$.
If $C$ is not smooth then we can approximate $C$ by nearby
smooth convex bodies $C_\varepsilon$, for $\varepsilon > 0$. By \cite[Theorem 3.1]{Fu}, the normal cycle
$N(C)$ is the Hausdorff limit of the manifolds $N(C_\varepsilon)$ for $\varepsilon \rightarrow 0$.
The normal cycle $N(C)$ is pure $(n-1)$-dimensional, and its smooth points are dense.

There are several other ways of defining the normal cycle. The one we like best 
uses the dual convex body $C^\vee$. Assuming that the origin is in the interior of $C$, 
\begin{equation}
\label{eq:makeidentification}
 N(C) \quad = \quad \bigl\{\,(u,v) \in \partial C \times \partial C^\vee \,
:\, v \cdot (u-u') \geq 0 \,\, \hbox{for all}\,\, u' \in C \,\bigr\}. 
\end{equation}
The normal cycle comes naturally with two surjective maps
$$ \pi_1 : N(C) \rightarrow \partial C, \,(u,v) \mapsto u \qquad {\rm and} 
\qquad \pi_2 : N(C) \rightarrow \partial C^\vee, \,(u,v) \mapsto v. $$
Let $\mathcal{E} \subseteq \partial C^\vee$ be the set of exposed points of $C^\vee$.
We have $v \in \mathcal{E}$ if and only if there exists $u \in C$ such that
$\pi_1^{-1}(u) = \{(u,v)\}$. 
A subset $\psi$ of $N(C)$ with $\pi_2(\psi) \subset \mathcal{E}$ is called a {\em patch} if
$\psi$ is a connected differentiable manifold,
${\rm dim}(\pi_1(\psi)) = n-1$, the fibers
of $\pi_2$ vary continuously in the Hausdorff metric,
and $\psi$ is maximal with these properties.
We say that $\psi$ is a $k$-patch if ${\rm dim}(\pi_2(\psi)) = n-k-1$.
This means that $\pi_2(\psi)$ is an $(n-k-1)$-dimensional manifold of exposed points of $C^\vee$,
and these exposed points support continuously varying  $k$-faces of $C$.

\begin{remark} \rm \label{rem:finite} If the trajectory $\mathcal{C}$ is algebraic then
its convex hull $\,C$ is semialgebraic. Also the normal cycle $N(C)$
 and all its patches $\psi$ are semialgebraic. This follows from Tarski's theorem on quantifier elimination, and
 we find that the number of patches of $C$ is finite.
We believe that finiteness holds more generally for compact trajectories.
But we do not yet know the precise statement.  Real analytic geometry
is much more delicate than real algebraic geometry.
For instance, the family of
semianalytic sets is not closed under projection. We refer to 
\cite{ABF} for a recent account. The concept of C-semianalytic sets,
introduced in \cite{ABF} and named after Cartan, might be appropriate
for our setting. One can hope that the convex trajectories and their
patches are C-semianalytic when $\phi$ in (\ref{eq:dynamics1}) is polynomial.
\end{remark}

We now study the following computational problem. The input is
a smooth curve $\mathcal{C}$ in $\RR^n$, typically arising as trajectory of a 
dynamical system (\ref{eq:dynamics1}). We seek~the boundary of the 
convex trajectory $C {=} {\rm conv}(\mathcal{C})$. The output is the list of all patches.

\begin{example}[$n=3$] \rm The convex body in Examples \ref{ex:yellow_green_1} 
and \ref{ex:yellowgreen} has six patches. It has two $2$-patches, namely the two
triangles. It has two irreducible edge surfaces (cf.~\cite{RS2}).
These have degrees $3$ and $16$. 
Each  contributes two $1$-patches to $\partial C$.
All six patches are  visible from the edges and
 triangles of the polytope in Figure~\ref{fig:yellow_green_1}.
\end{example}

\begin{example}[$n=4$] \rm
If $C$ is the convex hull of a curve $\mathcal{C}$ in $\RR^4$ then its
$1$-patches are surfaces of edges, its $2$-patches are
curves of $2$-faces, and its $3$-patches are the facets of $C$.
\end{example}

Our goal is to identify all patches of $C = {\rm conv}( \mathcal{C})$ from $\varepsilon$-approximations,
using the results in Section \ref{sec3}. We assume that $\mathcal{C}$ is a simplicial
curve, by which we mean that it satisfies the hypotheses (H1), (H2) and (H3), and the number
of patches of $C$ is finite (cf.~Remark \ref{rem:finite}).

The special case $n=2$ is
solved by Algorithm \ref{alg:boundary_examining_2}.
Algorithm \ref{alg:boundary_examining_3} computes the patches for $n \geq 3$.
We implemented this algorithm for $n=3$ and $n=4$.
A detailed theoretical analysis of this algorithm is left for future work.
The task is to identify the precise conditions under which the
output detects the true patches when applied to $\varepsilon$-approximations
with $\varepsilon \rightarrow 0$.

In the next section we report on some experiments with our methods. The code
 for dimensions $2$, $3$ and $4$ is made available at 
 $${\tt http://tools.bensolve.org/trajectories}.$$
 
 We discuss the steps in Algorithm \ref{alg:boundary_examining_3}. Step 1 is executed using {\tt Bensolve}
as discussed in Section \ref{sec4}. Each facet $H$  comes with its unit normal vector $v(H)$.
Step 2 reflects the conditions in the definition of patches. 
For instance,  the criterion $d(H_1,H_2) \leq \delta$, stating that $H_1$ and $H_2$ are
 close in Hausdorff distance, reflects the continuous variation of $k$-faces.
The exposed points in $\pi_2(\psi) \subset \mathcal{E}$ are represented
by the  vectors $v(H_i)$. The requirement that they are $\delta$-close along the 
edges of $G$ is our discrete version of the smoothness of $\psi$.
In step 3 we identify the connected components of $G$,
and these represent the patches of $C$.

The inner loop in steps 5--7 reflects our results in Section 3.
By Theorem \ref{th:1},  every $k$-face of $C$ is approximated by a facet $H \in G$.
Here, for each vertex of the $k$-face, the algorithm chooses a
nearby vertex $u_i$ of $H$. 
In the loop between steps 4 and 8, it can happen that 
a $\delta$-proximity cluster corresponds to more than one
vertex of the $k$-face. This happens for $k$-faces with an
 edge shorter than $\delta$.
For that reason, we take the maximum in step 9. 

Step 10 is very important and requires some explanation.
In a connected component of $G$, some facets will be misclassified:
they represent $k$-faces of $C$, but step 5 identifies less than $k+1$ 
proximity clusters at the fixed tolerance level $\delta$. For such facets,
step 10 adds additional points $u_i$ from an existing cluster to get
up to the correct value of $k$ for that patch.

\begin{algorithm_thm}(Detection of patches for $n \geq 3$)

\label{alg:boundary_examining_3}
\begin{algorithm}[H]
\SetKwData{CommonVerts}{CommonVerts}
\SetKwFunction{Conv}{Conv} \SetKwFunction{Graph}{Graph} 
\SetKwFunction{Ben}{Bensolve} \SetKwFunction{AggregateShort}{AggregateShort}
\SetKwFunction{ConnectedComponents}{ConnectedComponents}
\SetKwInOut{Input}{input}\SetKwInOut{Output}{output}
\Input{Finite list $\mathcal{A}$ of points on a curve $\mathcal{C}$
in $\RR^n$;  a threshold value $\delta > 0$} 
\Output{For each $k\geq 1$: the expected number $\#_k$ of $k$-patches of $C=\conv \left( \mathcal{C} \right)$ \\
For each $i$: list of $k$-polytopes that represent the $k$-patch $G_i$
}
Compute vertices $\mathcal{V}$, facets $\mathcal{H}$, incidence list $I_\mathcal{A}$ 
and adjacency list $A_\mathcal{A}$ of $\conv \left( \mathcal{A} \right)$. \\
Build a graph $G$ with node set  $\mathcal{H}$ as follows: two facets $H_1,H_2$ form an edge 
if  their unit normals $v(H_i)$   have distance $\leq \delta$,
$\,{\rm dim}(H_1 \cap H_2) = n-2$,  and $\,d(H_1,H_2) \leq \delta$.
\\
\ForEach{connected component $G_i $ of the graph $G$}{ 
\ForEach{facet $H \in G_i$}{
Find representatives $U=\{u_0,\ldots,u_k\}$ of the $\delta$-proximity clusters \\
of vertices of $H$ such that $F=\conv \left( U \right)$ is a $k$-face of $\conv \bigl(\mathcal{A} \bigr)$. \\
Associate the tuple $(u_0,\ldots,u_k;v)$ with that node of $G_i$.
}
$G_i$ represents a $k$-patch of $C$ if $k$ is the largest index encountered in the loop above. \\
Adjust all tuples with smaller indices found in step 7 to that common value of $k$.
}
Output $(\#_1,\ldots,\#_{n-1})$, where $\#_k$ is the number of graphs $G_i$
representing $k$-patches.
\end{algorithm}
\end{algorithm_thm}

\begin{example}[$n=4$] \rm
We illustrate the output of Algorithm~\ref{alg:boundary_examining_3}  when
$\mathcal{C}$ is a random trigonometric curve of degree six in $\RR^4$,
as in Section~\ref{sec6}, and $\mathcal{A} \subset \mathcal{C}$ is a finite approximation.
Figure \ref{fig:conn_comp_2} shows the graph $G$. Each node in $G$ is a face of the polytope ${\rm conv}(\mathcal{A})$. 
We find $\#_3 = 0$.
There are $\#_1 = 3$ patches for $k=1$, represented by the three connected components of $G$
  on the right in Figure~\ref{fig:conn_comp_2}.
These three connected graphs encode surfaces worth of edges.
The number of patches for $k=2$ is $\#_2 = 2$. These two components of $G$
are shown on the left in Figure \ref{fig:conn_comp_2}. Each node represents a triangle face of ${\rm conv}(\mathcal{C})$.
So, the picture on the left shows two curves of triangle faces in the boundary of our 
$4$-dimensional convex body. 
\end{example}

\begin{figure}[h]
\includegraphics[width=4cm]{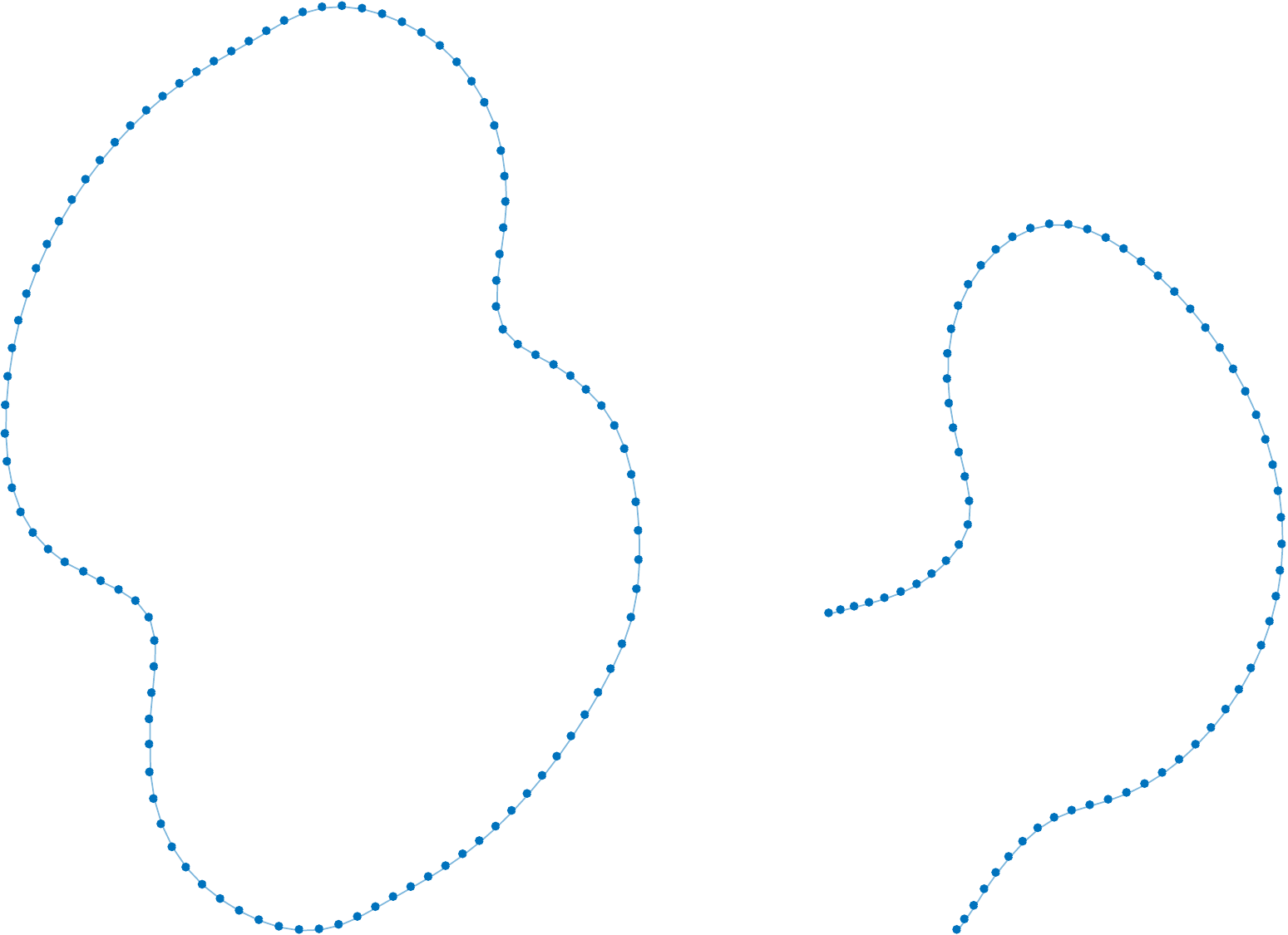} \qquad \qquad
\includegraphics[width=4.5cm]{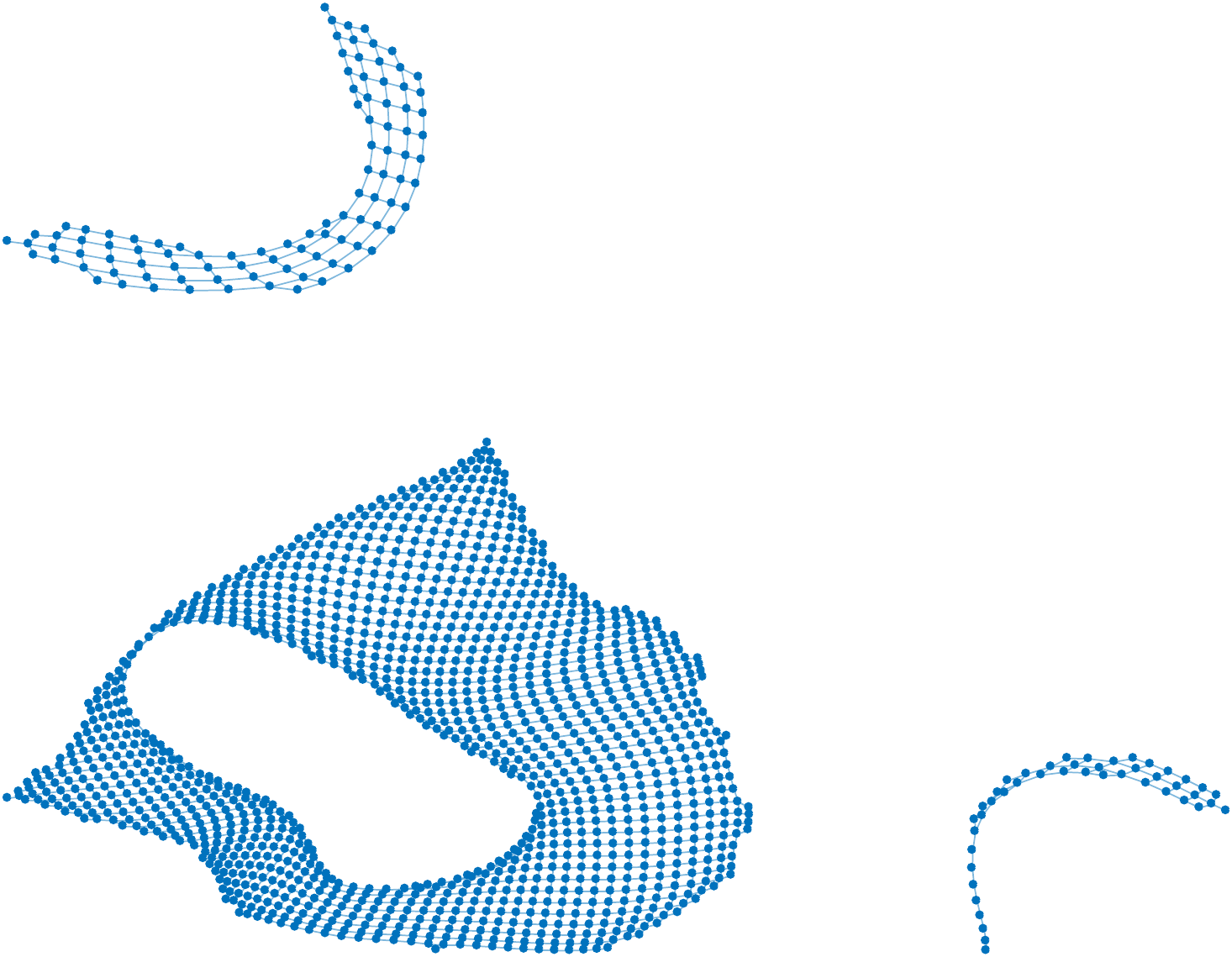}
\caption{\label{fig:conn_comp_2} 
Two $2$-patches (left) and three  $1$-patches (right) in the boundary of a 
$4$-dimensional convex body. It is the convex hull of a trigonometric curve of degree six.
The picture shows the graph $G$, with five connected components $G_i$, 
found by Algorithm~\ref{alg:boundary_examining_3}.
}
\end{figure}

\section{Algebraic and Trigonometric Curves}
\label{sec6}

In what follows we note that every algebraic curve can be realized locally
as the trajectory of an autonomous polynomial dynamical system.
This generalizes the Hamiltonian systems  (\ref{eq:hamiltonian}) seen in Section \ref{sec2}.
Hence, the computation of the convex hull of a real algebraic curve in $\RR^n$, discussed
in e.g.~\cite{RS2}, is a special case of the problem we addressed in Sections \ref{sec3}--\ref{sec5}.

Let $\mathcal{C}$ be an algebraic curve in $\RR^n$ and  let
$z $ be a regular point on $\mathcal{C}$. We construct an appropriate 
vector field $\phi(x)$ on $\RR^n$ as follows. Let
$f_1,f_2,\ldots,f_{n-1}$ be polynomials in $x = (x_1,x_2,\ldots,x_n)$
that cut out the curve $\mathcal{C}$ locally near its point~$z$.  Let $J$ denote their
Jacobian matrix. Thus $J$ is the $(n-1) \times n$ matrix whose entry in row $i$ and column $j$
is the partial derivative $\,\partial f_i / \partial x_j$. 
Let $J_i$ be $(-1)^{i+1}$ times the determinant of the
submatrix of $J$ obtained by deleting the $i$th column.
Fix the vector of polynomials $\phi = (J_1,J_2,\ldots,J_n)^T$.
Locally at $z$,  the kernel of $J$ is the line spanned by the vector $\phi$.
This follows from Cramer's rule, and it implies that $\phi(z)$ is a tangent vector to
the curve $\mathcal{C}$ at its point $z$.
We are interested in the dynamics of the system $\dot x = \phi(x)$ 
when the starting point is $z \in \mathcal{C}$.

\begin{proposition}
\label{prop:algsystem}
The trajectory of the dynamical system $\dot x = \phi(x) $ that starts at
a point $z$ on the algebraic curve $\mathcal{C}$ remains on the curve $\mathcal{C}$.
It either cycles around one nonsingular oval of $\,\mathcal{C}$, or it diverges towards infinity
in $\RR^n$, or it converges to a singular point of $\,\mathcal{C}$.
\end{proposition}

\begin{proof}
The proof is analogous to Corollary~\ref{cor:hamilton} which dealt with the case
$n~=~2$.
\end{proof}

\begin{example}[$n=3$] \rm
\label{ex:yellowgreen}  Let $\mathcal{C}$ 
be the trigonometric curve (\ref{eq:yg1}) in Example~\ref{ex:yellow_green_1}.
By \cite[\S 1]{RS2},  this is an algebraic curve, namely it is the zero set of the two polynomials
\begin{equation}
\label{eq:isthezeroset}
 f_1 \,= \, x^2 - y^2 - x z \quad {\rm and} \quad f_2 \,=\, z - 4 x^3 + 3 x . 
\end{equation}
With these two polynomials we associate the dynamical system 
\begin{equation}
\label{eq:yellowgreen}
\dot x = -2y \quad {\rm and} \quad
 \dot y = 12 x^3-5 x+z \quad {\rm and} \quad
\dot z = -24 x^2 y+6 y.
\end{equation}
Suppose we start this at a point on the curve $\mathcal{C}$, such as $(1,0,1)$.
The trajectory travels on $\mathcal{C}$ and it stops at the singular point $(0,0,0)$.
% In particular, the parametrization of $\mathcal{C}$ given by (\ref{eq:dynamics1})
% differs from that given by (\ref{eq:yg1}).
To get ${\rm conv}(\mathcal{C})$, we compute
 the convex hull of two trajectories obtained by using two different starting points on the curve given by (\ref{eq:isthezeroset}). The convex body has six patches, as shown in \cite[Figure 1]{RS2} and in our Figures
\ref{fig:yellow_green_1} and \ref{fig:yellowgreen}.

Using the methods  in Section \ref{sec7},
we analyzed the vector field on these facets and patches, and we found points with both
inward and outward pointing directions on each of them.
Figure \ref{fig:yellowgreen} shows a point in a triangle facet with outward pointing direction.
The resulting trajectory is also depicted. This solves Problem (ii) from the Introduction for this example.
\end{example}

\begin{figure}[t]
\begin{center}
 \includegraphics[width=0.3\textwidth]{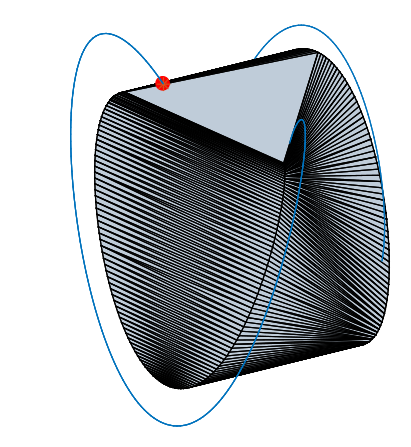}
 \caption{\label{fig:yellowgreen} The red boundary point 
 shows that the convex trajectory is not forward closed.}
    \end{center}
 \end{figure}

% Case studies such as that in Examples \ref{ex:yellow_green_1} and \ref{ex:yellowgreen}  
% are based on the numerical computation of the
% convex hull of a trajectory, given a sufficiently large, but finite, collection
% of sample points, as in Section \ref{sec4}. The true boundary of the convex hull is given by its patches.
% The facets of the approximating polytopes converge to the true boundary, as shown in Section \ref{sec3}.
% In Sections \ref{sec5} we demonstrated how the patches can be identified from the approximation.

Trigonometric curves also arise from linear dynamical systems.
Here (\ref{eq:dynamics1}) takes the form
$\dot x = A x $, where $A$ is a real $n \times n$-matrix.
We tested our convex hull algorithms  on linear systems for $n=3,4$.
We sampled matrices $A$ with no real eigenvalues.
This ensures that the trajectories are bounded in $\RR^n$.
They can be written in terms of trigonometric functions.
It was shown in \cite{Kai} that every convex trajectory of a linear 
dynamical system is forward-closed. Thus, computing the convex 
trajectory is equivalent to computing the attainable region.

Consider the {\em generalized moment curve},
whose convex hull was studied in \cite[Theorem 1]{Smi}.
Let $z = (1,0,1,0)$ and consider the linear dynamical system given by 
$$ A \,\,= \,\,2 \pi \cdot  \begin{small} \begin{pmatrix}
\, 0 & -p & \phantom{-}0 & \phantom{-}0 \,\, \\
\, p & \phantom{-}0 & \phantom{-}0 & \phantom{-}0 \, \,\\
\, 0 & \phantom{-}0 & \phantom{-}0 & -q \,\, \\
\, 0 & \phantom{-}0 & \phantom{-}q & \phantom{-}0 \, \,\end{pmatrix}\end{small} ,$$
where $p$ and $q$ are relatively prime positive integers. The trajectory is
the curve
$$ x(t) \,\, = \,\,\bigl( \,{\rm cos}( 2\pi p t)\,,\, {\rm sin}(2 \pi p t)\,,\,
{\rm cos}( 2\pi q t)\,,\, {\rm sin}(2 \pi q t)\, \bigr).
$$
The curve is closed, and we can restrict to $0 \leq t < 1$.
The convex hull of the curve is a $4$-dimensional convex body. 
By \cite[Theorem 1]{Smi}, there are no $3$-dimensional faces.
Assuming $p,q \geq 3$,  there are two $1$-patches and two $2$-patches.
The explicit description in \cite{Smi} makes this a useful test case.

\begin{example}[$p=3, q = 4$] \rm
\label{ex:smilansky} The  segment ${\rm conv}\{ x(s), x(t) \}$ is an edge
if and only~if 
$$ \frac{1}{4} < |s-t| < \frac{1}{3} \qquad
{\rm or} \qquad \frac{2}{3} < |s-t| < \frac{3}{4}. $$
In addition to this surface of edges, there are two curves of $2$-faces, namely the triangles
$$ \begin{matrix} {\rm conv} \bigl\{
x\bigl(t \bigr), x \bigl(t + \frac{1}{3} \bigr), x\bigl(t+ \frac{2}{3} \bigr) \bigr\}
\qquad {\rm for} \,\,\, 0 \leq t < \frac{1}{3} \end{matrix} $$
and the squares
$$ \begin{matrix} {\rm conv} \bigl\{
x\bigl(t \bigr), x \bigl(t + \frac{1}{4} \bigr), x\bigl(t+ \frac{1}{2} \bigr) ,
x \bigl(t + \frac{3}{4} \bigr) \bigr\}
\qquad {\rm for} \,\,\, 0 \leq t < \frac{1}{4} .\end{matrix} $$
These are all the exposed faces of the convex trajectory.
Even though the curve is not simplicial,  Algorithm \ref{alg:boundary_examining_3} works well, and
 we verified 
 % Example \ref{ex:smilansky} 
 Smilansky's findings using our software.  
\end{example}

We experimented with our {\tt Bensolve}-based code
for random trigonometric curves  $\, x: [0,1] \rightarrow \RR^n $.
The coordinates of $x$ are trigonometric polynomials of the form
$$ x_j(t) \,\,=\,\,\sum_{k=1}^d A_{jk} \cdot {\rm cos}(2 \pi k t) \,+\,
\,\sum_{k=1}^d B_{jk} \cdot {\rm sin}(2 \pi k t) \,+\,
C_{j} \qquad {\rm for} \,\,j=1,\ldots,n.
$$
We write the coefficients as a pair of $n \times d$ matrices $A$ and $B$
together with a column vector $C$, all filled with real numbers.
For general matrices, the resulting curve is an algebraic curve of degree $2d$ in $\RR^n$.
We computed many examples and recorded the features seen in the boundary.
We were most interested in the maximal numbers of facets that were observed.

\begin{table}[h]
 \setcounter{MaxMatrixCols}{15}
 \begin{footnotesize}
$$\begin{matrix} 
{\rm degree} \ $2d$ & 6 & 8 & 10 & 12 & {\bf 14} & 16 & 18 & 20 & 22 & 24 & 26 & 28 \\
\hbox{max $\#_2$} & 6 & 10 & 16  &17 & {\bf 20} & 21 & 24 & 26 & 28 & 30 & 34 &  34 \\
{\rm tritangents} & 8 & 80 & 280 & 672 & {\bf 1320} 
& 2288 & 3640 & 5440 & 7752 & 10640 & 14168 & 18400 \\
\hbox{max $\#_1$} & 10 & 16 & 24 & 26 & {\bf 30} 
& 32 & 35 & 38 & 41 & 44 & 46 & 50 \\
{\rm edge \ surface} &
30 & 70 & 126 & 198 & {\bf 286} & 390 & 510 & 646 & 798 & 966 & 1150 & 1350 
 \end{matrix} \vspace{-0.12in}
$$
\end{footnotesize}
\vspace{0.5cm}
\caption{\label{tab:3Ddata} Census of random trigonometric curves in $3$-space}
\end{table}

We sampled random data $(A,B,C)$ and  computed the convex hull
of the resulting curves. For $n=3$ we recorded the number
of triangles (= $2$-patches). The second row in Table \ref{tab:3Ddata} shows the maximal
number of triangles that was observed for given degree $2d$.
Each triangle spans a real tritangent plane of the curve.
The third row lists the number of complex tritangent planes for this curve,
which is a generic space curve of genus $0$ of degree $2d$.
The edge surface of the curve is an irreducible ruled surface that defines the
nonlinear part of the boundary of the convex hull. 
 Its degree is listed in the fifth row. This surface is the Zariski closure~of any of the $1$-patches. The fourth row shows the maximal number of
observed $1$-patches. The numbers in the
third and fifth row are taken from \cite[Corollary 3.1]{RS2}.

\begin{figure}[h]
\begin{center}
 \includegraphics[width=0.5\textwidth]{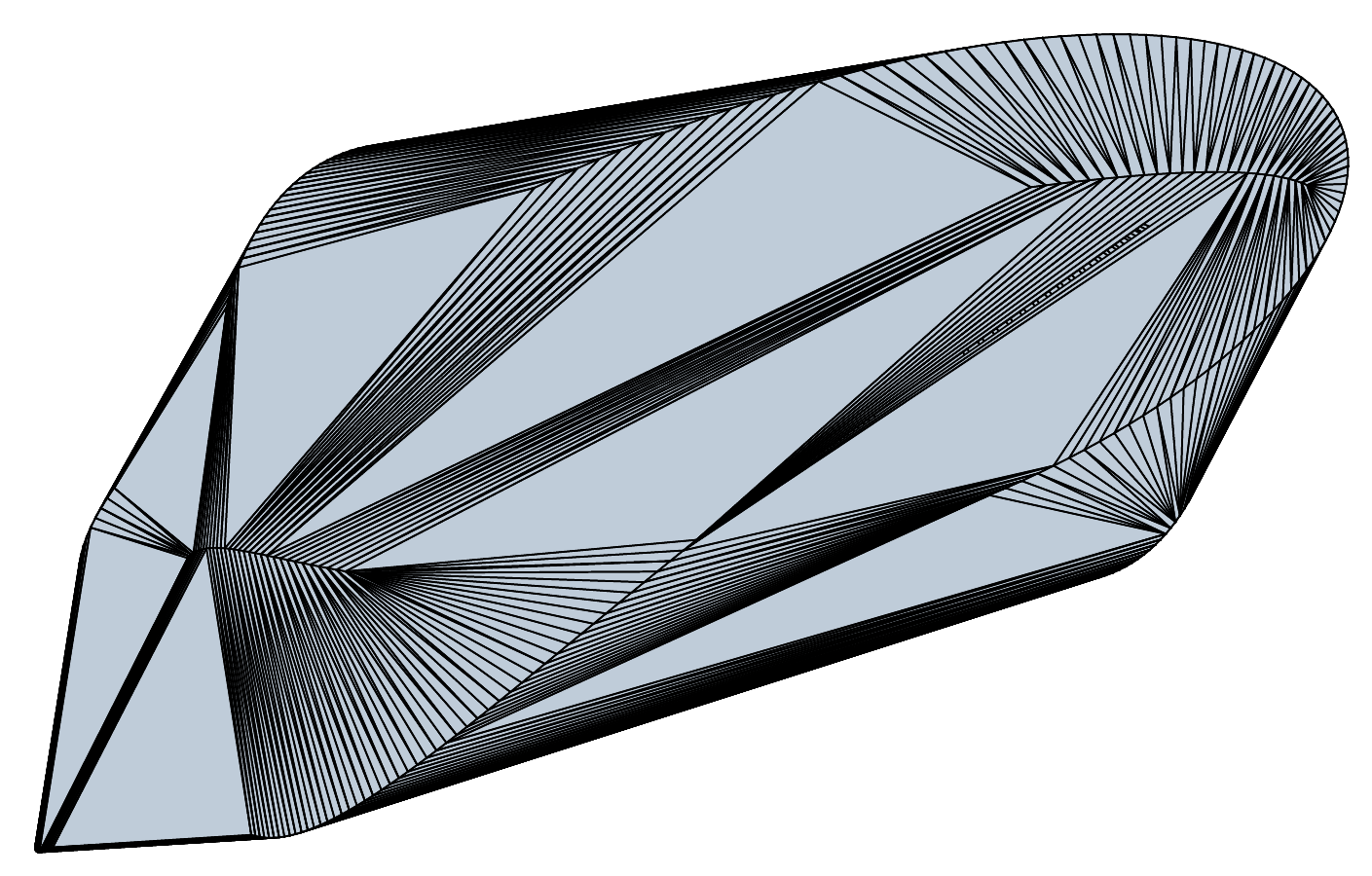}
 \caption{\label{fig:14} The convex hull of a trigonometric curve of degree $14$ in $3$-space.
 The boundary of this convex body consists 
  of triangles and of $1$-patches in a ruled surface of degree $286$.}
  \end{center}
 \end{figure}

We illustrate our computational results in Table \ref{tab:3Ddata}  for a curve
of degree $2d=14$.

\begin{example}[$2d=14$] \label{eq:winner} \rm
We consider the curve defined by
 the $3 \times 7$ matrices
$$ \qquad
A \,\, =\,\, \begin{footnotesize} \begin{pmatrix}
       \phantom{-}0.28561 & -0.024204 & -0.07664 &\, 0.43593 & \phantom{-}0.15244 & -0.24464 & \phantom{-}0.41538 \\
        -0.37439 & -0.30106 & \phantom{-}0.32118 & \,0.38410 &\phantom{-} 0.29990 & -0.14990 & -0.45481 \\
        -0.17997 & -0.16046 & -0.23522 &   \, 0.47912 & -0.08084 & \phantom{-}0.19628 & \phantom{-}0.46895 \\
\end{pmatrix} \end{footnotesize}
$$
$$ {\rm and} \quad
B \,\, =\,\, \begin{footnotesize}\begin{pmatrix}
      -0.39109 & \,0.06742 & -0.12451 & \,0.44073 & -0.20822 & -0.03646 & -0.01034 \\
  \phantom{-}0.48646 & \,0.38580 & -0.13216 & \,0.36184 & \phantom{-}0.30633 & -0.14131 &
  \phantom{-} 0.48650 \\
        -0.15326 &\, 0.32591 & \phantom{-}0.02569 & \,0.23351 & -0.34972 & \phantom{-}0.04772 & \phantom{-}0.42441 \\
\end{pmatrix} \end{footnotesize}, $$
along with the vector
$\, C \, =\, \begin{pmatrix}   0.39768 &   0.42346 &   0.23797 \end{pmatrix}^T$.
The convex hull of this curve has $20$ triangle facets. 
It is shown in Figure \ref{fig:14}. The planes that define the triangles
are tritangent planes.  The curve is generic
and has $1320$ tritangent planes over $\CC$. The nonlinear part of the
boundary is the edge surface \cite{RS2}. This is an irreducible
ruled surface of degree $286$. 
% It contributes the $1$-patches in the boundary seen in Figure \ref{fig:14}.
\end{example}

\section{Partitioning the Boundary}
\label{sec7}

Problem (i) from the Introduction was addressed in the previous sections.
In this section we propose a solution to Problem (ii). Our input now is the output
of Algorithm \ref{alg:boundary_examining_2} or Algorithm \ref{alg:boundary_examining_3}.
If $n=3$ then we are given all $2$-patches (triangles) and all $1$-patches
(in the edge surface) of the convex hull $C= {\rm conv}(\mathcal{C})$ of 
a trajectory $\mathcal{C}$ of the dynamical system
(\ref{eq:dynamics1}). 
For instance, the output of Algorithm \ref{alg:boundary_examining_3} might be
Figure~\ref{fig:31patches}. This is a variant of Figure~\ref{fig:14}, derived from a trigonometric
curve $\mathcal{C}$ of degree $14$. Here $\mathcal{C}$ is simplicial, and the boundary $\partial C$
consists of $20$ triangles (= $2$-patches) and $30$ $1$-patches, shown in different colors in
Figure~\ref{fig:31patches}.

\begin{figure}[h]
\begin{center}
 \includegraphics[width=0.35\textwidth]{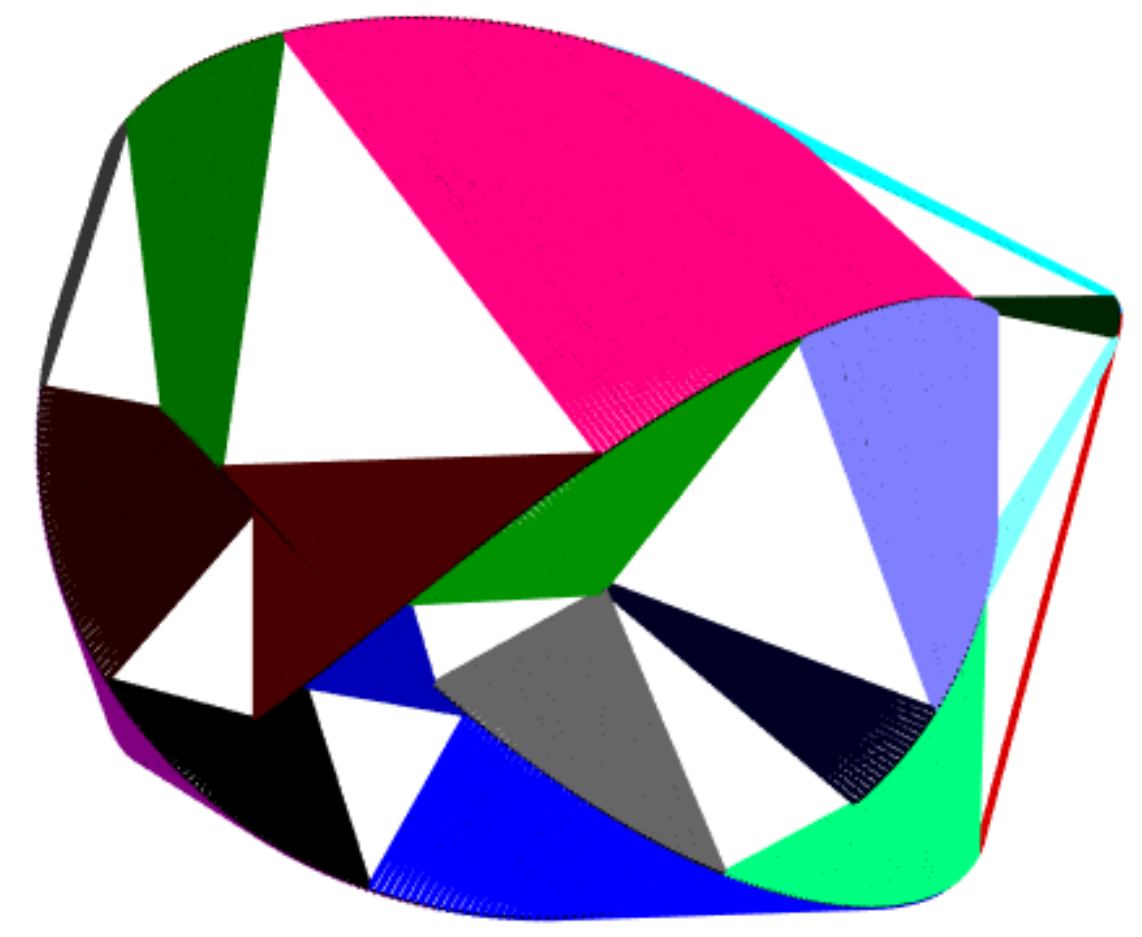}
 \caption{\label{fig:31patches} The convex hull of a trigonometric curve of degree $14$.}
  \end{center}
 \end{figure}

We seek to partition the boundary $\partial C$ into two regions. In one region, the vector
$\phi(z)$ points inward and in the other it points outward, as in Figure~\ref{fig:yellowgreen}.
 For $n=2$ this partition is determined by  the formulas (\ref{eq:part2a}) and (\ref{eq:part2b}).
In what follows we generalize this method to $n \geq 3$. 

For a pair $(u,v)$ in the normal cycle $N(C)$, $v$ is an outward pointing unit normal vector at $C$.
 The right hand side of (\ref{eq:dynamics1}) {\em points inward} at $u \in \partial C$
if $ \phi(u) \cdot v \leq 0$ for all $v \in \pi_2(\pi_1^{-1}(u))$. It {\em points outward} otherwise. 
Taking into account that the set $\{ u \in \RR^n \mid \pi_2(\pi_1^{-1}(u))= \{v\}\}$ is dense in $\partial C$, we see that the boundary between
inward and outward pointing vectors $\phi(z)$ is the image under the projection
$\pi_1: N(C) \rightarrow \partial C$ of the set
$ \big\{\, (u,v) \in N(C) \,\, :\,\, \phi(u) \cdot v \,= \, 0 \,\bigr\}$.
This image has dimension $n-2$ inside the $(n-1)$-dimensional normal cycle $N(C)$.

Let $\psi$ be a $k$-patch of $C$. The output of Algorithm \ref{alg:boundary_examining_3} represents $\psi$ by a connected graph $G_i$. Each node of $G_i$
is a face $F=\conv \{u_0,\ldots,u_{k}\}$ along with a normal vector $v$ at $C$.
 We are interested in the restriction of the boundary above to the patch $\psi$ of interest:
\begin{equation}
\label{eq:uvpatch}
 \pi_1\left(\big\{\, (u,v) \in \psi \,\, :\,\, \phi(u) \cdot v \,= \, 0 \,\bigr\}\right). 
 \end{equation}
We obtain an approximate representation of $\partial C$ employing Algorithm \ref{alg:boundary_examining_3}.
Algorithm \ref{alg:redgreen} computes a partition of this approximation into inward and outward pointing regions. The simplex $\Delta_k$
is the convex hull of the unit vectors  in $\mathbb{R}^{k}$.

Consider a fixed graph $G_i$ in the loop started in step 1. The value of $k$ is 
constant in step 2. This is ensured by step 10 in Algorithm \ref{alg:boundary_examining_3}.
The graph $G_i$ represents a $k$-patch of $C$. We algorithmically realize the restriction of the hypersurface (\ref{eq:uvpatch}) to the $k$-faces in that patch in step 4.
If $k=1$ then this results in a finite partition of a line segment.
For $k=2$ we obtain a curve in a triangle, and for $k=3$ we obtain
a surface in a tetrahedron. The latter case happens only for $n \geq 4$.

The paradigm for our computations is the algebraic case.
Suppose that $\mathcal{C}$ is an algebraic curve, for instance
obtained from a dynamical system as in Proposition \ref{prop:algsystem}.
In that case, the equation $\phi(u) \cdot v = 0$ is a polynomial in  $k$ 
unknowns $\lambda_1,\ldots,\lambda_k$, after setting $\lambda_0 = 1-\sum_{j=1}^k \lambda_j$.
To be precise, let $ v= (v_1,\ldots,v_n)$ be the normal vector
of the $k$-face in question. In the situation of Proposition \ref{prop:algsystem}, the
polynomial equation we are solving on $\Delta_k$ takes the form
$$ \phi(u) \cdot v \,\,=\,\,
 \sum_{l=1}^n  J_l \bigl( \lambda_0 u_0 + \cdots + \lambda_k u_k \bigr) \cdot v_l \,\,\, = \,\,\, 0 . $$
In some situations, we know the equation $f=0 $ of the hypersurface $\pi_1(\psi)$ in $\RR^n$.
Here $f$ is analytic or polynomial, depending on the instance. With this, we~write
$$ 
v_l \,\, = \,\, \frac{\partial{f}}{\partial x_l} \bigl( \lambda_0 u_0 + \cdots + \lambda_k u_k \bigr)
\qquad {\rm for} \,\,\, l = 1,2,\ldots,n . $$
This formula allows us to solve the equation  $\phi(u) \cdot v = 0$ 
simultaneously on the entire and exact $k$-patch, and not just on each approximated $k$-face of $\psi$
individually, as it is done in step 4 of Algorithm \ref{alg:redgreen}.

\begin{algorithm_thm}{(Partitioning the boundary of a convex trajectory)}
\label{alg:redgreen}

\begin{algorithm}[H]
\SetKwInOut{Input}{input}\SetKwInOut{Output}{output}
\Input{The graphs $G_i$ representing the patches of a convex trajectory of (\ref{eq:dynamics1})}
\Output{Partition of the boundary into inward and outward pointing regions}
\ForEach{connected graph $G_i$ in the output of Algorithm \ref{alg:boundary_examining_3}}{
\ForEach{node $(\{u_0,\ldots,u_{k}\},v)$ of the graph $\,G_i\,$}{
Set $u = \sum_{i=0}^{k} \lambda_j u_j$ where $\lambda_j $ are nonnegative unknowns
satisfying $\sum_{j=0}^{k} \lambda_j = 1 $. \\
Compute the $(k-1)$-dimensional hypersurface in $\Delta_{k}$ defined by $\phi(u) \cdot v = 0$ and identify inward and outward pointing regions.
}
}
\end{algorithm}
\end{algorithm_thm}

% We illustrate this computation and how it differs from Algorithm \ref{alg:redgreen} in an example.

\begin{example}[$n=3$] \rm
We partition the boundary of the convex body in Figure~\ref{fig:yellowgreen}.
Its edge surface has two irreducible components, of degrees $3$ and $16$.
Each contributes two patches.
The cubic is $f_2 = z - 4x^3 + 3x$ in (\ref{eq:isthezeroset}).
The degree $16$ polynomial $g$ is displayed in \cite[\S 1]{RS2}.
On the cubic patches, the equation $\nabla f_2 \cdot \phi = 0$ holds
identically, so these patches are not partitioned. Hence,
every trajectory that starts on a cubic patch remains in that patch.
 The two degree $16$ patches are partitioned by a curve of degree $262$,
 obtained by intersecting the patches with the surface defined by
 $ \nabla g \cdot \phi = 0$. The two triangle facets lie in the planes
 $z  =\pm 1$. They are partitioned by the lines $y=0$ and $x=\pm 1/2$.
 Figure~\ref{fig:yellowgreen} shows the trajectory that starts at a
 red point in the outward pointing region of the top triangle.
 \end{example}

In the next section we apply our methods  to partition the boundary of
convex trajectories of dynamical systems that arise from chemical reaction networks.
Figure~\ref{fig:weakrev}  shows our partition for a convex trajectory arising in an
application.

\section{Chemical Reaction Networks}
\label{sec8}

Our interest in convex trajectories and attainable regions is motivated
by dynamical systems  for chemical reaction networks. 
The aim of {\em attainable region theory} \cite{MGHGM} is to design
  chemical reactors that are optimal for chemical reactions of interest.
This research topic was pioneered by Feinberg and Hildebrand in \cite{FH}.
They argued that 
optimal reactors are often 
 found at the extreme points of the attainable region and showed that these extreme points are realizable
  by parallel operations of elementary reactor types.
We refer to  \cite{Kai} for a recent study in the setting of convex algebraic geometry \cite{BPT, RS2}.
The notion of {\em protrusions} in \cite[\S 2.6]{FH} is
dual to our notion of {\em patches} in Section \ref{sec5}, in the sense that a $k$-patch on $C$
corresponds to an $(n-k-1)$-dimensional protrusion on $C^\vee$ under the
self-duality of~the normal cycle $N(C)$ in (\ref{eq:makeidentification}).
Patches and protrusions are interesting  for further research.

To explain the connection to chemical reactions, we
work in the setting of {\em mass action kinetics}.
Let $G$ be a directed graph with $m$ vertices, each labeled by a 
monomial $x^{{\bf a}_i}$ in $n$ unknowns 
$x = (x_1,\ldots,x_n)$. These unknowns are the concentrations of
$n$ chemical species. The $m$ monomials are the chemical complexes.
Each $x_j = x_j(t)$ is a function of time $t$.
With each edge $(i,j)$ of $G$, connecting two
 monomials $x^{{\bf a}_i}$ and $x^{{\bf a}_j}$, we associate a parameter $\kappa_{ij}$, which is
the rate constant for that reaction. The associated dynamical system is given~by
\begin{equation}
\label{eq:CRN}
\phi(x) \,=\, \bigl(x^{{\bf a}_1}, x^{{\bf a}_2},\ldots,x^{{\bf a}_m}\bigr)
\cdot \Lambda_G(\kappa) \cdot \bigl( {\bf a}_1,{\bf a}_2,\ldots,{\bf a}_m \bigr)^T. 
\end{equation}
This is a row vector of length $n$, written as a product of three matrices, of formats
$1 \times m$, $\, m \times m$, and $m \times n$. The middle matrix $\Lambda_G(\kappa)$ is
the {\em Laplacian} of the graph $G$, with entry $\kappa_{ij}$ for each edge
$(i,j)$, all other off-diagonal entries set to zero, and diagonal entries inferred so that 
the row sums of $\Lambda_G(\kappa)$  are zero. An important special case of
(\ref{eq:CRN}) are the toric dynamical systems  \cite{CDSS}. 
For details see the forthcoming book by Dickenstein 
and Feliu~\cite{DF} and Shiu's dissertation \cite[\S 1.3]{Shiu}.

A feature of many chemical reaction systems (\ref{eq:CRN}) is the existence of
{\em conservation relations}. These arise when the entries of the 
polynomial vector $\phi(x)$ are linearly dependent over $\RR$.
If this happens then all trajectories lie in certain lower-dimensional
subspaces of $\RR^n$. In such cases, the ambient dimension $n$ can be reduced.
Namely, we always transform our dynamical systems so that each trajectory affinely spans
$\RR^n$. We examined some important classes with $n \leq 4$.
Of particular interest are chemical reaction networks that admit multistationarity.
The smallest~such networks were characterized by Joshi and Shiu~\cite{JS}.

Given an arbitrary polynomial dynamical system (\ref{eq:dynamics1}), it is natural
to ask whether it arises from some chemical reaction network $G$. The solution to
this inverse problem was given by H$\acute{\text{a}}$rs  and T$\acute{\text{o}}$th \cite{HT}.
They showed that $\phi = (\phi_1,\ldots,\phi_n)$ is realized by
a graph $G$ as above if and only if each monomial  with negative coefficient in $\phi_i$ is divisible by $x_i$, 
for $i=1,2,\ldots,n$.

We discussed Hamiltonian systems in Section \ref{sec2}.
The following result characterizes chemical reaction 
dynamics in $\RR^2$ that  is Hamiltonian. It would be interesting to
study such reaction networks, along with the higher-dimensional 
versions arising from Proposition \ref{prop:algsystem}.

\begin{proposition} \label{prop:hamiltonianCRN}
Let $h(x,y)$ be a polynomial. The Hamiltonian system
(\ref{eq:hamiltonian}) can~be realized as a mass action system (\ref{eq:CRN}) 
if and only if  the coefficients of all powers of $y$ in $h(x,y)$ are
 nonnegative and those of all pure powers of $x$ are nonpositive,~i.e.
 $$ \qquad h(x,y) \,\, = \,\, xy \cdot a(x,y) \, - \, b(x) \,+\, c(y), $$
{where $b$ and $c$ have nonnegative coefficients}. 
 \end{proposition}

\begin{proof}
This is immediate from Theorem 3.2 in \cite{HT}.
\end{proof}

The convex trajectory we compute as an answer to problem (i) in the Introduction is a first approximation to the attainable region
  and its representation in  \cite{FH}. In Section \ref{sec7} we presented  an algorithm for partitioning the approximated boundary of
the convex trajectory. We next apply that algorithm to two interesting chemical reaction networks. 

\begin{example}[$n=4, m=5$] \label{ex:vdv} \rm
We revisit the {\em Van de Vusse reaction}. This is~studied extensively in the chemistry literature (cf.~\cite[Chapter~6]{MGHGM}). 
The network equals
\begin{center}
\schemestart 
    \subscheme{$X_2$}
    \arrow(be--ac){<-[1]}[180]
    \subscheme{$X_1$}
      \arrow(@be--d){->[1]}
    \subscheme{$X_3$}
\schemestop
\end{center}
\begin{center}
\schemestart 
    \subscheme{$X_4$ \, .}
    \arrow(be--ac){<-[10]}[180]
    \subscheme{2$X_1$}
\schemestop
\end{center}
The rate constants $\kappa_{ij}$ are written over the edges.
 The mass action system  (\ref{eq:CRN})~equals
\begin{equation}\label{eqn:ds}
\phi(x) \,\,=\,\,
\begin{bmatrix}
 x_{1} & x_2 & x_{3} & x_1^2 & x_4
 \end{bmatrix}\cdot \begin{small}  \begin{bmatrix}
-1	 &  1  &  0 &  0  &  0\\
 0	 & -1  &  1 &  0  &  0\\
 0	 &  0  &  0 &  0  &  0\\
 0	 &  0  &  0 & -10  & 10\\
 0  &  0  &  0 & 0 & 0
 \end{bmatrix}\cdot \begin{bmatrix}
1&0&0&0 \\
0&1&0&0\\
0&0&1&0\\
2&0&0&0\\
0&0&0&1 \end{bmatrix}, \end{small}
\end{equation}
where $x_i$ is the concentration of species $X_i$.
Explicitly, this is the system (\ref{eq:dynamics1}) with
$$ \phi(x_1,x_2,x_3,x_4) \,\,\, = \,\,\,
\bigl[\,-x_1 - 20 x_1^2\,,\,x_1 - x_2\,, \,x_2 \,, \,10 x_1^2\, \bigr]. $$
Computations of the critical reactors of this system are found in
\cite[Section 5.3]{MGHGM}.

Fix the starting point $y = (1,0,0,0)$. The trajectory starting at $y$ converges
to the steady state $y^* = (0,0,0.1522, 0.4238)$.
 The dynamics takes place in $\RR^4$, but the stoichiometry space has dimension $3$.
 In our analysis we use the projection onto the first three coordinates. 
 With this, the trajectory is an arc in a $3$-dimensional space, shown in blue in Figure~\ref{fig:weakrev}.

\begin{figure}[h]
\begin{center}
 \includegraphics[width=0.29\textwidth]{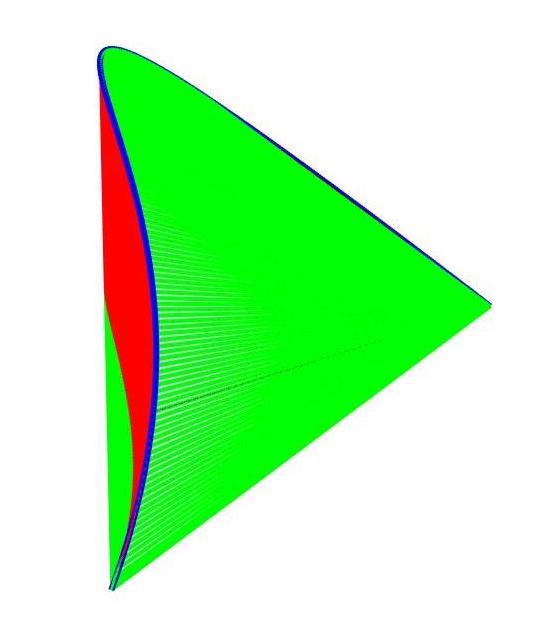} \qquad 
 \includegraphics[width=.5\linewidth]{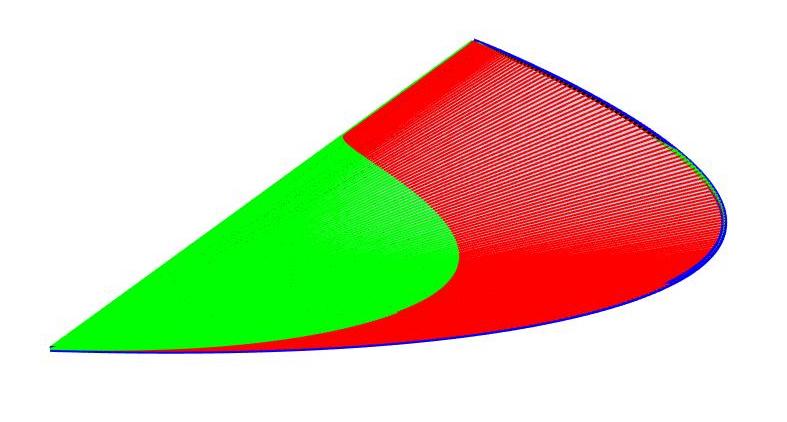}
 \caption{\label{fig:weakrev} Convex trajectory of the Van de Vusse reaction and the partition of its boundary.}
   \end{center}
 \end{figure}
 
We computed the convex trajectory starting at $y$
for (\ref{eqn:ds}) using Algorithm \ref{alg:boundary_examining_3},
and we then partitioned its boundary using Algorithm \ref{alg:redgreen}.
The result is shown in Figure \ref{fig:weakrev}.
The convex body has two $1$-patches, obtained by joining each of the two endpoints 
with each point on the curve. One of the patches is entirely green. This means that
the vector field is pointing inward on that patch. The other patch is partitioned into a 
green region and a red region, as shown on the right in Figure~\ref{fig:weakrev}.
Red color indicates that the vector field points outward. In particular, the convex trajectory is
strictly contained in the attainable region.
\end{example}

The mass action system (\ref{eq:CRN}) is called {\em weakly reversible} if every connected
component of the underlying directed graph is strongly connected in $G$, i.e.~there is a directed path
from any node in the component to any other node.
It was conjectured in \cite{Kai} that convex trajectories of weakly reversible systems
are forward closed. We here resolve that conjecture.

\begin{proposition}
Not all convex trajectories of weakly reversible systems are forward closed. 
\end{proposition}

The proof is by  computation using our algorithms.
Here is the counterexample:

\begin{example}[Weakly Reversible System]\rm
Consider the following weakly reversible network
\begin{center}

\schemestart 
    \subscheme{2$X_1$ + $X_2$}
    \arrow(be--ac){<=>[2][2]}[180]
    \subscheme{$X_1$ + $X_3$}
      \arrow(@be--d){<=>[4][4]}
    \subscheme{2$X_2$ + $X_3$}
    \arrow(@d--f){<=>[2][4]}
    \subscheme{$X_1$ + $X_2.$}
    % \arrow(@ac--@f){<=>[$\kappa_7$][$\kappa_8$]}[90]
\schemestop
\end{center}
%The corresponding polynomial dynamical system (\ref{eq:CRN}) has the right hand side
%\begin{gather*}
% \phi(x) \,\,\,=\,\,\,\begin{bmatrix}
% x_{2}^{2} x_3 & x_1x_2 & x_{1}^{2} x_2 & x_1x_3
% \end{bmatrix} \cdot
%\begin{bmatrix}
%-6	 &  2  &  4 &  0 \\
% 4	 & -4  &  0 &  0 \\
% 4	 &  0  & -6 &  2 \\
% 0	 &  0  &  2 & -2 
% \end{bmatrix} \cdot
% \begin{bmatrix}
%0&2&1 \\
%1&1&0\\
%2&1&0\\
%1&0&1 \end{bmatrix}.
%\end{gather*}
The three coordinates for (\ref{eq:dynamics1}) are explicitly given by
$$
\begin{matrix}
\phi_1 & = & -10x_1^2x_2 + 10x_2^2x_3 - 4x_1x_2 + 2x_1x_3, \\
\phi_2 & = & 2x_1^2x_2 - 6x_2^2x_3 + 4x_1x_2 + 2x_1x_3, \\
\phi_3 & = & 6x_1^2x_2 - 6x_2^2x_3 + 4x_1x_2 - 2x_1x_3.
\end{matrix}
$$
This system has deficiency zero, and it is a toric dynamical system \cite{CDSS}.
There are no conservation relations. The trajectories are curves that span the ambient space~$\RR^3$.

\begin{figure}[h]
\begin{center}
 \includegraphics[width=0.47\textwidth]{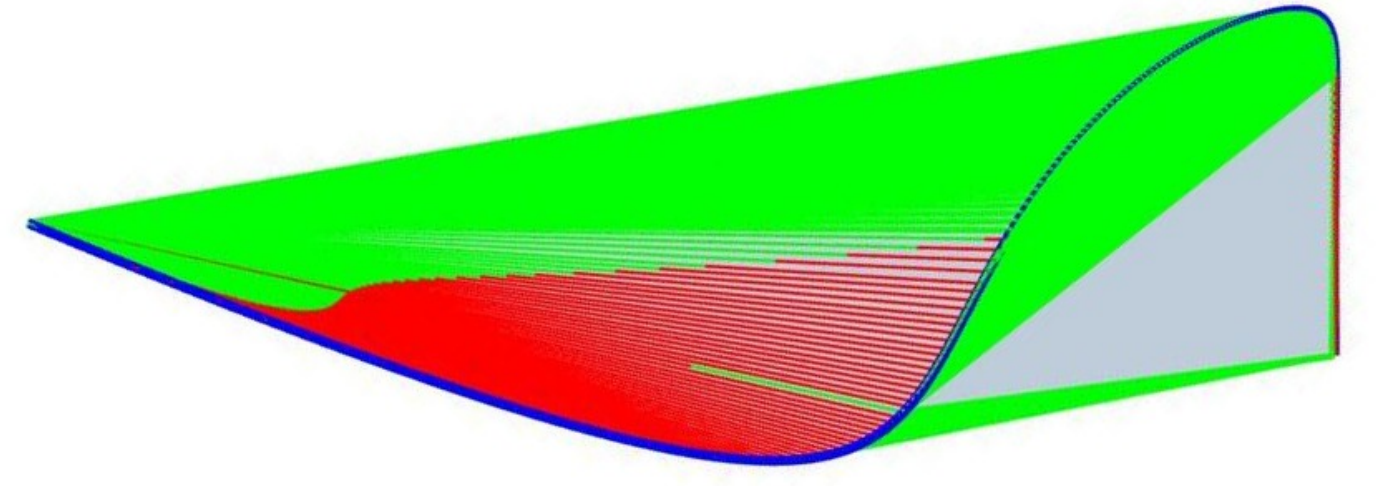}
 \caption{\label{fig:weakrev2} Convex trajectory of a weakly reversible system that is not forward closed.}
  \end{center}
 \end{figure}

Let $y=(4,4,2)$. The convex body $C={\rm convtraj}(y)$ 
was computed using Algorithm \ref{alg:boundary_examining_3}
and  is shown in Figure~\ref{fig:weakrev2}.
The triangle shown in gray is a $2$-patch of $C$.
The vector field given by $(\phi_1,\phi_2,\phi_3)$ points inward at all points on 
that triangle facet. 
We also show the partition of the $1$-patches of $C$, as computed by Algorithm~\ref{alg:redgreen}.
One of the patches is partitioned into a green region and a red region.
As before, the vector field points outward at each red point.
We conclude that the convex trajectory $C$ of $y$ is not forward closed.
\end{example}

\bigskip

\noindent
{\bf Acknowledgments.}
Nidhi Kaihnsa was funded by the
International Max Planck Research School {\em Mathematics in the Sciences} (IMPRS). 
 Bernd Sturmfels was partially supported by the
 US National Science Foundation (DMS-1419018). 
 
 \bigskip \bigskip

\begin{small}

\end{small}
\bigskip \bigskip

\noindent
\footnotesize {\bf Authors' addresses:}

\smallskip

\noindent Daniel Ciripoi, Universit\"at Jena
\hfill {\tt daniel.ciripoi@uni-jena.de}

\noindent Nidhi Kaihnsa, MPI-MiS Leipzig
\hfill {\tt kaihnsa@mis.mpg.de}

\noindent Andreas L\"ohne, Universit\"at Jena
\hfill {\tt andreas.loehne@uni-jena.de}

\noindent Bernd Sturmfels,
 \  MPI-MiS Leipzig and
UC  Berkeley \hfill  {\tt bernd@mis.mpg.de}


\begin{thebibliography}{10}

\setlength{\itemsep}{-0.4mm}





\bibitem{ABF}
F.~Acquistapace, F.~Broglia and J.F.~Fernando:
{\em On globally defined semianalytic sets},
Mathematische Annalen {\bf 366} (2016), 613--654.


\bibitem{BK}
C.~Bajaj and M.~S.~Kim:
{\em Convex hulls of objects bounded by algebraic curves},
Algorithmica {\bf 6} (1991), 533--553. 


\bibitem{BPT}
G.~Blekherman, P.~Parrilo and R.~Thomas:
{\em Semidefinite Optimization and Convex Algebraic Geometry},
 MOS-SIAM Series on Optimization {\bf 13}, 2012.
 
% \bibitem{BKSW} P.~Breiding, S.~Kalisnik, B.~Sturmfels and M.~Weinstein:
% {\em Learning algebraic varieties from samples}, 
% Revista Matem\'atica Complutense {\bf 31} (2018) 545-593.

% \bibitem{B} H.~P.~Benson:
% {\em An outer approximation algorithm for generating all efficient extreme
%  points in the outcome set of a multiple objective linear programming problem}, J.~Global~Optim. {\bf 13} (1998) 1--24.



\bibitem{Bro}
E.M.~Bronstein:
{\em Approximation of convex sets by polytopes},
J.~Math.~Sciences {\bf 153} (2008), 727--762.

\bibitem{CLW} D.~Ciripoi, A.~L\"ohne and B.~Wei{\ss}ing:
{\em Bensolve tools - Calculus of convex polyhedra, calculus of polyhedral convex functions,
global optimization, vector linear programming for Octave and Matlab}, Version 1.2,
{\tt http://tools.bensolve.org}.

\bibitem{CDSS} G.~Craciun, A.~Dickenstein, A.~Shiu and B.~Sturmfels:
{\em Toric dynamical systems}, J.~Symbolic Comput.~{\bf 44} (2009), 1551-1565.

\bibitem{DF} A.~Dickenstein and E.~Feliu:
{\em Algebraic Methods for Biochemical Reaction Networks}, 
textbook in preparation.

\bibitem{ELS} M.~Ehrgott, A.~L\"ohne and L.~Shao:
{\em A dual variant of {B}enson's ``outer approximation algorithm'' for multiple objective linear programming}, J. Global Optim. {\bf 52} (2012), 757--778.


\bibitem{EKKL} G.~ Elber, M.~S.~Kim, Y.~J.~Kim and J.~Lee  
{\em Efficient convex hull computation for planar freeform curves},
Computers \& Graphics {\bf 35} (2011), 698--705.


\bibitem{FH} M.~Feinberg and D.~Hildebrandt: {\em Optimal reactor design from a geometric
viewpoint--l. Universal properties of the attainable region}, Chemical Engin.~Science, {\bf 52} (1997), 1637-1665.

% \bibitem{Fei1} M.~Feinberg: {\em Optimal reactor design from a geometric
% viewpoint. Part-Il. Critical Sidestream Reactors}, Chemical Engineering Science, {\bf 55} (2000) 2455-2479.

\bibitem{Fu} J.~Fu: {\em Algebraic integral geometry},
Integral Geometry and Valuations, 47--112, Adv. Courses Math. CRM Barcelona, Birkh\"auser/Springer, Basel, 2014. 

\bibitem{HT}V.~H$\acute{\text{a}}$rs  and J.~T$\acute{\text{o}}$th: {\em On the inverse problem of reaction kinetics}, 
Colloquia  Math.~Societatis
J$\acute{\text{a}}$nos Bolyai {\bf 30}, {\em Qualitative Theory of Differential Equations},  363--379, Szeged, 1979.

\bibitem{Hoe}
J.~van der Hoeven: {\em Certifying Trajectories of Dynamical Systems}, 
In: I.S.~Kotsireas, S.M.~Rump and C.K.~Yap (eds.) 
Mathematical Aspects of Computer and Information Sciences. MACIS 2015. Lecture Notes in Computer Science {\bf 9582} (2016), 
520--532.

\bibitem{JS} B.~Joshi and A.~Shiu: {\em Which small reaction networks are multistationary?},
SIAM J.~Appl.~Dyn.~Syst.~{\bf 16} (2017), 802--833.

\bibitem{Kai} N.~Kainhsa: {\em Attainable regions of dynamical systems},
presented at MEGA 2019 (Effective Methods in Algebraic Geometry), Madrid,
June 2019, {\tt arXiv:1802.07298}.

% \bibitem{LL}
% C.~Lassez and J.-L. Lassez.
% {\em Quantifier elimination for conjunctions of linear constraints via a
%  convex hull algorithm}, In B.~R. Donald, D.~Kapur, and J.~L. Mundy (eds.): {\em Symbolic
%  and Numerical Computation for Artificial Intelligence}, 103--119,
%  Academic Press,   1993.

\bibitem{LW2} A.~L\"ohne and B.~Wei\ss ing:
{\em Equivalence between polyhedral projection, multiple objective linear programming 
and vector linear programming}, Math.~Methods~Oper.~Res. {\bf 84} (2016) 411--426.
 
\bibitem{LW} A.~L\"ohne and B.~Wei\ss ing:
{\em The vector linear program solver Bensolve -- notes on theoretical background},
 European J.~Oper.~Res.~{\bf 260} (2017), 807--813.

\bibitem{MGHGM} D.~Ming, D.~Glasser, D.~Hildebrandt, B.~Glasser and M.~Metzger:
{\em Attainable Region Theory: An Introduction to Choosing an Optimal Reactor},
John Wiley and Sons, 2016.

% \bibitem{RS1} K.~Ranestad and B.~Sturmfels:
% {\em The convex hull of a variety}, in P.~Br\"anden, M.~Passare and M.~Putinar:
% {\em Notions of Positivity and the Geometry of Polynomials}, 331-344,
%Trends in Mathematics, Springer-Verlag, Basel, 2011.

\bibitem{RS2} K.~Ranestad and B.~Sturmfels:
{\em On the convex hull of a space curve}, Advances in Geometry {\bf 12} (2012), 157-178. 


%\bibitem{SW} A.~A.~Sch\"{a}ffer and C.~J.~Van Wyk:
%{\em Convex hulls of piecewise-smooth Jordan curves}, J.~of Algorithms {\bf 8} (1987) 66--94.


\bibitem{Shiu} A.~Shiu:
{\em Algebraic methods for biochemical reaction network theory}, 
Ph.D. thesis, 2010.
  
\bibitem{Smi} Z.~Smilansky: {\em Convex hulls of generalized moment curves},
 Israel J.~Math.~{\bf 52} (1985), 115--128.


\bibitem{Zie} G.~Ziegler: {\em Lectures on Polytopes},
   Grad.~Texts in Math~{\bf 152}, Springer, New York, 1995.

\end{thebibliography}
\end{document}